\documentclass[11pt]{article}
\usepackage{latexsym,amsmath,url,caption2,epsfig}
\usepackage{textcomp}
\usepackage{amsfonts,euscript}
\usepackage{graphicx}
\usepackage{amsmath}
\usepackage{amssymb}
\usepackage{amsfonts}
\usepackage{amsthm}
\usepackage{mathrsfs}
\usepackage[all]{xy}
\usepackage{epstopdf}
\usepackage{subfigure}
\usepackage{rotating}
\usepackage{appendix}
\usepackage{multirow,rotating}
\usepackage{url}
\usepackage{subfigure}
\usepackage{setspace}
\usepackage{listings}
\usepackage{color}
\usepackage{tikz}
\usepackage{graphics}
\usepackage[breaklinks=true]{hyperref}
\usepackage[english]{babel}
\usepackage{datetime}
\usepackage{subfigure}
\usepackage{booktabs}
\usepackage{rotfloat}
\usepackage[numbers]{natbib}
\usepackage{algorithm}
\usepackage{algorithmicx}
\usepackage{algpseudocode}
\usepackage{framed}

\graphicspath{{pic/}}
\newtheorem{theorem}{Theorem}

\newtheorem{lemma}{Lemma}
\newtheorem{proposition}{Proposition}

\theoremstyle{definition}

\algdef{SE}[DOWHILE]{Do}{doWhile}{\algorithmicdo}[1]{\algorithmicwhile\ #1}

\allowdisplaybreaks

\begin{document}

\title{Optimal Uncertainty Size in Distributionally Robust Inverse
Covariance Estimation}
\date{ This Version: \today }
\author{Jose Blanchet \and Nian Si}
\date{\emph{Management Science and Engineering, Stanford University}}
\maketitle

\begin{abstract}
In a recent paper, Nguyen, Kuhn, and Esfahani (2018) built a
distributionally robust estimator for the precision matrix of the Gaussian
distribution. The distributional uncertainty size is a key ingredient in the
construction of this estimator. We develop a statistical theory which shows
how to optimally choose the uncertainty size to minimize the associated
Stein loss. Surprisingly, rather than the expected canonical square-root
scaling rate, the optimal uncertainty size scales linearly with the sample
size.
\end{abstract}

\section{Introduction}

Motivated by a wide range of problems which require the estimation of the
inverse of a covariance matrix, \cite{nguyen2018distributionally} recently
constructed an estimator based on distributionally robust optimization using
the Wasserstein distance in Euclidean space. A crucial ingredient is the
distributional uncertainty size, which plays the role of a regularization
parameter.

In their paper, \cite{nguyen2018distributionally} show excellent empirical
performance of their estimator in comparison to several commonly used
estimators (based on shrinkage and regularization). The comparison is based
in terms of the corresponding Stein loss (defined in terms of the
likelihood, as we shall review). However, no theory is provided as how to
choose the distributional uncertainty size.

Our goal is to provide an asymptotically optimal expression for the
distributional uncertainty size, in terms of the Stein loss performance, as
the sample size increases.

This paper provides interesting insights which validate the empirical
observations in \cite{nguyen2018distributionally}. In particular, in the
Introduction of \cite{nguyen2018distributionally}, leading to equation (4),
they argue that the distributional uncertainty size, $\rho _{n}$, should
scale at rate $\rho _{n}=O\left( n^{-1/2}\right) $ (where $n$ is the sample
size) due to the existence of a central limit theorem for the Wasserstein
distance for Gaussian distributions. However, the numerical experiments,
reported in Section 6.1 of \cite{nguyen2018distributionally}, suggest an
optimal scaling of the form $\rho _{n}=O\left( n^{-\kappa }\right) $ where $%
\kappa >1/2$.

Our main result shows that the asymptotically optimal choice of
distributional uncertainty is of the form $\rho _{n}=$ $\rho _{\ast
}n^{-1}\left( 1+o\left( 1\right) \right) $ as $n\rightarrow \infty $, where $%
\rho _{\ast }>0$ is a constant which is characterized explicitly. Our
results therefore validate the empirical findings of \cite%
{nguyen2018distributionally} with $\kappa =1$.

This paper is organized as follows. We review the estimator of \cite%
{nguyen2018distributionally} and state our main result in Section \ref%
{Section_Main}. We then provide the proof of our result in Section \ref%
{Proof_of_Thm_1}. Numerical experiments are included in Section \ref%
{Sec_Numerical}, which provide a sense of the non-asymptotic performance of
our asymptotically optimal choice.

\section{Basic Notions and Main Result\label{Section_Main}}

We now review the basic definitions underlying the estimator from \cite%
{nguyen2018distributionally}. Suppose we have i.i.d. samples $\xi _{i}\sim
\mathcal{N}\left( 0,\Sigma _{0}\right) $ (normally distributed with zero
mean and covariance matrix $\Sigma _{0}$), where $\xi _{i}\in \mathbb{R}^{d}$
and $\Sigma _{0}$ is assumed to be strictly positive definite. We write
\begin{equation*}
\hat{\Sigma}_{n}=\frac{1}{n}\sum_{i=1}^{n}\xi _{i}\xi _{i}^{T},
\end{equation*}%
and let $\mathbb{\hat{P}}_{n}$ correspond to a distribution with mean zero
and covariance matrix $\hat{\Sigma}_{n}$, which we denote as $\mathcal{N}%
\left( 0,\hat{\Sigma}_{n}\right) $. Throughout our development we use the
notation $\left\langle A,B\right\rangle =$ tr$(A^{T}B)$ for any $d\times d$
matrices $A$, $B$, where $A^{T}$ denotes the transpose of $A$. The identity
matrix is denoted by $I$. We use $\Rightarrow $ and $\overset{p}{\rightarrow
}$ to denote weak convergence (convergence in distribution) and convergence
in probability, respectively. Finally, for two symmetric matrices $A,B\in
\mathbb{R}^{d\times d}$, $A\preceq B$ denotes that $A-B $ is positive
semi-definite.

We define the Stein loss as
\begin{equation*}
L(X,\Sigma _{0})=-\log \det (X\Sigma _{0})+\left\langle X,\Sigma
_{0}\right\rangle -d,
\end{equation*}%
where $X$ is any estimator of the precision matrix (i.e. the inverse
covariance matrix).

Given an uncertainty size $\rho $, let us write $X_{n}^{\ast }(\rho )$ for
the distributionally robust estimator proposed in \cite%
{nguyen2018distributionally}; i.e,%
\begin{equation}
X_{n}^{\ast }(\rho )=\arg \min_{X\succ 0}\left\{ -\log \det X+\sup_{\mathbb{Q%
}\in \mathcal{P}_{\rho }}\mathbb{E}^{\mathbb{Q}}\left[ \left\langle \xi \xi
^{T},X\right\rangle \right] \right\} ,  \label{DRO_sol}
\end{equation}%
where $\mathcal{P}_{\rho }$ is the set of $d$-dimensional normal
distributions with mean zero and which lie within distance $\rho $ measured
in the Wasserstein sense, which we define next; see, for example, Chapter 7
in \cite{villani2003topics} for background on Wasserstein distances and,
more generally, optimal transport costs. The Wasserstein distance (more
precisely, the Wasserstein distance of order two with Euclidean norm) is
defined as follows. First, let $\mathcal{M}_{+}(\mathbb{R}^{d}\times \mathbb{%
R}^{d})$ be the set of Borel (positive) measures on $\mathbb{R}^{d}\times
\mathbb{R}^{d}$ and define the Wasserstein distance between $\mathbb{\hat{P}}%
_{n}$ and $\mathbb{Q}$ via%
\begin{eqnarray*}
\mathbb{W}_{2}(\mathbb{\hat{P}}_{n},\mathbb{Q})=\inf_{\pi \in \mathcal{M}%
_{+}(\mathbb{R}^{d}\times \mathbb{R}^{d})}\left\{ \left( \int \left\Vert
z-w\right\Vert _{2}^{2}\pi \left( \mathrm{d}x,\mathrm{d}w\right) \right)
^{1/2}\right.  && \\
:\int_{w\in \mathbb{R}^{d}}\pi \left( \mathrm{d}x,\mathrm{d}w\right) =%
\mathbb{\hat{P}}_{n}\left( \mathrm{d}x\right) , &&\left. \int_{x\in \mathbb{R%
}^{d}}\pi \left( \mathrm{d}x,\mathrm{d}w\right) =\mathbb{Q}\left( \mathrm{d}%
w\right) \right\} .
\end{eqnarray*}%
Then
\begin{equation*}
\mathcal{P}_{\rho }=\left\{ \mathbb{Q}\sim \mathcal{N}\left( 0,\Sigma
\right) \text{ for some }\Sigma :\mathbb{W}_{2}(\mathbb{\hat{P}}_{n},\mathbb{%
Q})\leq \rho \right\} .
\end{equation*}

In simple terms, $\mathcal{P}_{\rho }$ is the set of probability measures
corresponding to a Gaussian distribution which lie within $\rho $ units in
the Wasserstein distance from $\mathbb{\hat{P}}_{n}$. It is well known (in
fact, an immediate consequence of the delta method) that $n^{1/2}\mathbb{W}%
_{2}(\mathbb{\hat{P}}_{n},\mathbb{P}_{\infty })\Rightarrow \mathbb{W}$ for
some limit law $\mathbb{W}$ which can be explicitly characterized (but not
important for our development; see \cite{rippl2016limit}). This result
suggests that $\rho :=\rho _{n}$ should scale in order $O\left(
n^{-1/2}\right) $. It is therefore somewhat surprising that the optimal
scaling of $\rho $ for the purpose of minimizing the Stein loss is actually
significantly smaller, as the main result of this paper indicates next.

\begin{theorem}
\label{Thm_Main}Let
\begin{equation}
\rho _{n}=\arg \min_{\rho \geq 0}\{\mathbb{E}[L(X_{n}^{\ast }(\rho ),\Sigma
_{0})]\},  \label{rho_n_def}
\end{equation}%
then
\begin{equation*}
\lim_{n \rightarrow \infty} n\rho _{n}=\rho _{\ast },
\end{equation*}%
for $\rho _{\ast }>0$.
\end{theorem}

\textbf{Remark:} The explicit expression of $\rho _{\ast }$ can be
characterized as follows. First, let us consider the weak limit%
\begin{equation*}
Z=\lim_{n\rightarrow \infty }n^{1/2}\left( \hat{\Sigma}_{n}-\Sigma
_{0}\right) ,
\end{equation*}%
which, by the Central Limit Theorem is a matrix with correlated mean zero
Gaussian entries. Then, we have
\begin{equation*}
\rho _{\ast }=\mathbb{E}\left( \frac{4\text{tr}\left( \Sigma
_{0}^{-2}Z\Sigma _{0}^{-1}Z\right) }{\text{tr}(\Sigma _{0}^{-1})^{1/2}}-%
\frac{\text{tr}(Z\Sigma _{0}^{-2})^{2}}{\text{tr}(\Sigma _{0}^{-1})^{3/2}}%
\right) \frac{\text{tr}(\Sigma _{0}^{-1})}{4\text{tr}(\Sigma _{0}^{-2})}.
\end{equation*}%
Theorem \ref{Thm_Main} indicates that $\rho _{\ast }>0$, which will be
verified as a part of the proof of this result.

\section{Proof of Theorem \protect\ref{Thm_Main}\label{Proof_of_Thm_1}}

We first collect the following observations, which we summarize in the form
of propositions and lemmas for which we provide references or corresponding
proofs in the appendix \cite{supplement}. We then use these results to
develop the proof of Theorem \ref{Thm_Main}.

\subsection{Auxiliary Results}

We provide a lemma based on the analytical solution (Theorem 3.1 in \cite%
{nguyen2018distributionally}). {\color{blue} }

\begin{lemma}
\label{lemma} When $n>d$ and $\rho \leq 1,$ with probability one, we have
following Taylor expansions%
\begin{eqnarray*}
\partial X_{n}^{\ast }(\rho )/\partial \rho  &=&\hat{A}_{n}+O(\rho ), \\
X_{n}^{\ast }(\rho )^{-1} &=&\hat{\Sigma}_{n}-\hat{\Sigma}_{n}\hat{A}_{n}%
\hat{\Sigma}_{n}\rho +O(\rho ^{2}),
\end{eqnarray*}%
where%
\begin{equation*}
\hat{A}_{n}=-\frac{2}{\sqrt{\text{tr}(\hat{\Sigma}_{n}^{-1})}}\hat{\Sigma}%
_{n}^{-2},
\end{equation*}%
Furthermore, the remainder terms satisfy%
\begin{eqnarray}
\frac{\partial X_{n}^{\ast }(\rho )}{\partial \rho }-\hat{A}_{n} &\succeq
&-\left( \frac{4\hat{M}_{n}+2\hat{M}_{n}^{2}}{\sqrt{\sum_{i=1}^{d}\hat{%
\lambda}_{i}^{-1}}}\sum_{i=1}^{d}\frac{\hat{v}_{i}\hat{v}_{i}^{T}}{\hat{%
\lambda}_{i}^{2}}\right) \rho   \notag \\
\frac{\partial X_{n}^{\ast }(\rho )}{\partial \rho }-\hat{A}_{n} &\preceq
&\left( \frac{2\hat{M}_{n}^{3}+8\hat{M}_{n}}{\sqrt{\sum_{i=1}^{d}\hat{\lambda%
}_{i}^{-1}}}\sum_{i=1}^{d}\frac{\hat{v}_{i}\hat{v}_{i}^{T}}{\hat{\lambda}%
_{i}^{2}}\right) \rho ,  \label{remainder}
\end{eqnarray}%
and
\begin{equation}
-\left( \frac{2\left( 1+\hat{M}_{n}\right) ^{2}}{\sum_{i=1}^{d}\hat{\lambda}%
_{i}^{-1}}\sum_{i=1}^{d}\frac{\hat{v}_{i}\left( \hat{v}_{i}\right) ^{T}}{%
\hat{\lambda}_{i}}\right) \rho ^{2}\preceq X_{n}^{\ast }(\rho )^{-1}-\hat{%
\Sigma}_{n}+\hat{\Sigma}_{n}\hat{A}_{n}\hat{\Sigma}_{n}\rho \preceq \left(
\frac{2\hat{M}_{n}}{\sqrt{\sum_{i=1}^{d}\hat{\lambda}_{i}^{-1}}}%
\sum_{i=1}^{d}\hat{v}_{i}\left( \hat{v}_{i}\right) ^{T}\right) \rho ^{2},
\label{remainder2}
\end{equation}%
where
\begin{equation*}
\hat{M}_{n}=\frac{8}{\left( \min_{i}\hat{\lambda}_{i}\right) \min \left\{ d,%
\frac{\sqrt{d}}{\sqrt{\max_{i}\hat{\lambda}_{i}}}\right\} }.
\end{equation*}
\end{lemma}

From Lemma \ref{lemma}, we have that
\begin{equation}
X_{n}^{\ast }(\rho )^{-1}-\Sigma _{0}=\left( \hat{\Sigma}_{n}-\Sigma
_{0}\right) -\hat{\Sigma}_{n}\hat{A}_{n}\hat{\Sigma}_{n}\rho +O(\rho ^{2}).
\label{rho_decomp2}
\end{equation}

The first proposition provides standard asymptotic normality results for
various estimators.

\begin{proposition}
\label{probCLT}The following convergence results hold

\begin{itemize}
\item[(1)] $\frac{1}{\sqrt{n}}\sum_{i=1}^{n}\xi _{i}\Rightarrow N(0,\Sigma
_{0}),$

\item[(2)] $\sqrt{n}\left( \hat{\Sigma}_{n}-\Sigma _{0}\right) \Rightarrow Z$%
, where $Z$ is a symmetric matrix of jointly Gaussian random variables with
mean zero and
\begin{equation*}
cov(Z_{i_{1}j_{1}},Z_{i_{2}j_{2}})=\mathbb{E}\xi ^{(i_{1})}\xi ^{(j_{1})}\xi
^{(i_{2})}\xi ^{(j_{2})}-\left( \mathbb{E}\xi ^{(i_{1})}\xi
^{(j_{1})}\right) \left( \mathbb{E}\xi ^{(i_{2})}\xi ^{(j_{2})}\right)
=\sigma _{i_{1},i_{2}}^{2}\sigma _{j_{1},j_{2}}^{2}{}+\sigma
_{i_{1},j_{2}}^{2}\sigma _{j_{1},i_{2}}^{2},
\end{equation*}%
where $\xi ^{(i)}$ is the i-th entry of $\xi $ and $\sigma
_{i,j}^{2}=cov(\xi ^{(i)}\xi ^{(j)})$.

\item[(3)] $\hat{A}_{n}\overset{p}{\rightarrow }A_{0}$ and $\sqrt{n}\left(
\hat{A}_{n}-A_{0}\right) \Rightarrow Z_{A}$, where $A_{0}=-\frac{2}{\sqrt{%
\text{tr}(\Sigma _{0}^{-1})}}\Sigma _{0}^{-2}$ and
\begin{equation*}
Z_{A}=-\frac{\text{tr}\left( \Sigma _{0}^{-1}Z\Sigma _{0}^{-1}\right) \Sigma
_{0}^{-2}}{\text{tr}(\Sigma _{0}^{-1})^{3/2}}+2\frac{\Sigma _{0}^{-1}Z\Sigma
_{0}^{-2}+\Sigma _{0}^{-2}Z\Sigma _{0}^{-1}}{\text{tr}(\Sigma
_{0}^{-1})^{1/2}}.
\end{equation*}
\end{itemize}
\end{proposition}

Further, we also have the following observations.

\begin{proposition}
\label{large_zero}

\begin{itemize}
\item[(1)] $\mathbb{E}\left\langle Z,Z_{A}\right\rangle >0$,

\item[(2)] $\mathbb{E}\left\langle \hat{\Sigma}_{n}-\Sigma _{0},\hat{A}%
_{n}-A_{0}\right\rangle >0.$
\end{itemize}
\end{proposition}

\begin{lemma}
\label{temp} The following convergence in expectation results hold

\begin{itemize}
\item[(1)] $\mathbb{E}\left[ \left\langle \hat{\Sigma}_{n}\hat{A}_{n}\hat{%
\Sigma}_{n},\hat{A}_{n}\right\rangle \right] \rightarrow \left\langle \Sigma
_{0}A_{0}\Sigma _{0},A_{0}\right\rangle .$

\item[(2)] $\mathbb{E}\left\langle \sqrt{n}\left( \hat{\Sigma}_{n}-\Sigma
_{0}\right) ,\hat{A}_{n}\right\rangle \rightarrow \mathbb{E}\left\langle
Z,A_{0}\right\rangle .$

\item[(3)] $\mathbb{E}\left\langle \sqrt{n}\left( \hat{\Sigma}_{n}-\Sigma
_{0}\right) ,\sqrt{n}\left( \hat{A}_{n}-A_{0}\right) \right\rangle
\rightarrow \mathbb{E}\left\langle Z,Z_{A}\right\rangle .$
\end{itemize}
\end{lemma}

The following proposition shows consistency of the estimator.

\begin{proposition}
\label{consistency} For $\rho _{n}$ defined in (\ref{rho_n_def}), we have $%
\lim_{n\rightarrow \infty }\rho _{n}=0.$
\end{proposition}

Using the previous technical results we are ready to provide the proof of
Theorem \ref{Thm_Main}.

\subsection{Development of Proof of Theorem \protect\ref{Thm_Main}}

The gradient of the Stein loss is given by
\begin{equation*}
h(X,\Sigma _{0})=\frac{\partial L(X,\Sigma _{0})}{\partial X}=-X^{-1}+\Sigma
_{0}.
\end{equation*}%
We claim that $\rho _{n}=0$ is not a minimizer. The derivative of loss
function with respect to $\rho $ evaluating at $\rho =0$ is
\begin{equation*}
\left. \frac{\partial L(X_{n}^{\ast }(\rho ),\Sigma _{0})}{\partial \rho }%
\right\vert _{\rho =0}=\left\langle -\hat{\Sigma}_{n}+\Sigma _{0},\hat{A}%
_{n}\right\rangle .
\end{equation*}%
And by Proposition \ref{large_zero}, we have
\begin{equation*}
\mathbb{E}\left\langle -\hat{\Sigma}_{n}+\Sigma _{0},\hat{A}%
_{n}\right\rangle =-\mathbb{E}\left\langle \hat{\Sigma}_{n}-\Sigma _{0},\hat{%
A}_{n}-A_{0}\right\rangle <0,
\end{equation*}%
which shows that $\rho _{n}=0$ is not a minimizer. Furthermore, we have $%
\lim_{\rho \rightarrow \infty }L(X_{n}^{\ast }(\rho _{n}),\Sigma
_{0})=+\infty $ (see, Proposition 3.5 in \cite{nguyen2018distributionally}).
Therefore, the optimal solution is an interior point, i.e., $\rho _{n}\in
(0,+\infty ).$ Since $\rho _{n}$ is chosen to minimize $\mathbb{E}\left[
L(X_{n}^{\ast }(\rho _{n}),\Sigma _{0})\right] $, we have that $\rho _{n}$
satisfies the first order condition
\begin{equation}
\mathbb{E}\left\langle h(X_{n}^{\ast }(\rho _{n}),\Sigma _{0}),\hat{A}%
_{n}+O(\rho _{n})\right\rangle =0.  \label{FOC}
\end{equation}%
By plugging (\ref{rho_decomp2}) into (\ref{FOC}), we have%
\begin{equation}
\mathbb{E}\left\langle h(X_{n}^{\ast }(\rho _{n}),\Sigma _{0}),\hat{A}%
_{n}+O(\rho _{n})\right\rangle =-\mathbb{E}\left\langle \hat{\Sigma}%
_{n}-\Sigma _{0}-\hat{\Sigma}_{n}\hat{A}_{n}\hat{\Sigma}_{n}\rho _{n}+O(\rho
_{n}^{2}),\hat{A}_{n}+O(\rho _{n})\right\rangle =0,  \label{imp_eqn}
\end{equation}%
which is equivalent to
\begin{equation}
\mathbb{E}\left[ \left\langle \hat{\Sigma}_{n}\hat{A}_{n}\hat{\Sigma}_{n},%
\hat{A}_{n}\right\rangle \right] \rho _{n}+O(\rho _{n}^{2})=\mathbb{E}%
\left\langle \hat{\Sigma}_{n}-\Sigma _{0},\hat{A}_{n}+O(\rho
_{n})\right\rangle .  \label{FOC2}
\end{equation}%
The validity of expanding the expectations follows by applying the uniform
integrability results of the upper and lower bounds in (\ref{remainder}) and
(\ref{remainder2}) underlying the proof of Lemma \ref{temp}.

Now, note that, also by Lemma \ref{temp},%
\begin{equation*}
\lim_{n\rightarrow \infty }\mathbb{E}\left[ \left\langle \hat{\Sigma}_{n}%
\hat{A}_{n}\hat{\Sigma}_{n},\hat{A}_{n}\right\rangle \right] =\left\langle
\Sigma _{0}A_{0}\Sigma _{0},A_{0}\right\rangle =4\text{tr}(\Sigma _{0}^{-2})/%
\text{tr}(\Sigma _{0}^{-1})>0.
\end{equation*}%
By multiplying $\sqrt{n}$ on both sides of (\ref{FOC2}) and by Slutsky's
lemma (Theorem 1.8.10 in \cite{lehmann2006theory}), we have%
\begin{equation*}
\lim_{n\rightarrow \infty }\sqrt{n}\left( \mathbb{E}\left[ \left\langle \hat{%
\Sigma}_{n}\hat{A}_{n}\hat{\Sigma}_{n},\hat{A}_{n}\right\rangle \right] \rho
_{n}+O(\rho _{n}^{2})\right) =\lim_{n\rightarrow \infty }\mathbb{E}%
\left\langle \sqrt{n}\left( \hat{\Sigma}_{n}-\Sigma _{0}\right) ,\hat{A}%
_{n}+O(\rho _{n})\right\rangle =\mathbb{E}\left\langle Z,A_{0}\right\rangle
=0.
\end{equation*}%
The last equality follows from $\mathbb{E}Z=0$ and $A_{0}$ being
deterministic. Therefore,%
\begin{equation*}
\lim_{n\rightarrow \infty }\sqrt{n}\rho _{n}=0.
\end{equation*}%
Furthermore, since $\mathbb{E}\left[ \hat{\Sigma}_{n}-\Sigma _{0}\right] =0$
for every $n,$ we have (once again by Lemma \ref{temp})
\begin{eqnarray*}
&&\lim_{n\rightarrow \infty }\mathbb{E}\left\langle n\left( \hat{\Sigma}%
_{n}-\Sigma _{0}\right) ,\hat{A}_{n}+O(\rho _{n})\right\rangle  \\
&=&\lim_{n\rightarrow \infty }\mathbb{E}\left\langle n\left( \hat{\Sigma}%
_{n}-\Sigma _{0}\right) ,\hat{A}_{n}-A_{0}+A_{0}+O(\rho _{n})\right\rangle
\\
&=&\lim_{n\rightarrow \infty }\left\langle n\mathbb{E}\left[ \hat{\Sigma}%
_{n}-\Sigma _{0}\right] ,A_{0}\right\rangle +\lim_{n\rightarrow \infty }%
\mathbb{E}\left\langle \sqrt{n}\left( \hat{\Sigma}_{n}-\Sigma _{0}\right) ,%
\sqrt{n}\left( \hat{A}_{n}-A_{0}\right) +O(\sqrt{n}\rho _{n})\right\rangle
\\
&=&\mathbb{E}\left\langle Z,Z_{A}\right\rangle .
\end{eqnarray*}%
By multiplying $n$ on both sides of (\ref{FOC2}), we have%
\begin{equation*}
\rho _{\ast }=\lim_{n\rightarrow \infty }n\rho _{n}=\frac{\mathbb{E}%
\left\langle Z,Z_{A}\right\rangle }{\left\langle \Sigma _{0}A_{0}\Sigma
_{0},A_{0}\right\rangle }=\mathbb{E}\left( \frac{4\text{tr}\left( \Sigma
_{0}^{-2}Z\Sigma _{0}^{-1}Z\right) }{\text{tr}(\Sigma _{0}^{-1})^{1/2}}-%
\frac{\text{tr}(Z\Sigma _{0}^{-2})^{2}}{\text{tr}(\Sigma _{0}^{-1})^{3/2}}%
\right) \frac{\text{tr}(\Sigma _{0}^{-1})}{4\text{tr}(\Sigma _{0}^{-2})}>0,
\end{equation*}%
which is the desired result.

\section{Numerical Experiments\label{Sec_Numerical}}

Here we provide various numerical experiments to provide an empirical
validation of our theory and the performance of the asymptotically optimal
choice of uncertainty size in finite samples.

The first example is in one dimension. The data is sampled from a normal
distribution, $N(0,\sigma _{0}^{2})$; i.e, $\Sigma _{0}=\sigma _{0}^{2}$ in
the real line. Therefore,
\begin{equation*}
A_{0}=-2\sigma _{0}^{-3}\text{, }\mathbb{E}\left\langle Z,Z_{A}\right\rangle
=6\sigma _{0}^{-1}.
\end{equation*}%
Theorem \ref{Thm_Main} indicates that
\begin{equation*}
\lim_{n\rightarrow \infty }n\rho _{n}=\frac{3}{2}\sigma _{0}.
\end{equation*}

In our numerical example we fix $\sigma _{0}^{2}=10$. We vary the number of
data points, $n$, ranging from 10 to 1000. For each $n$, we use $T=5000$
trials to compute empirically the optimal choice of $\rho =\rho _{n}$ in
order to minimize the empirical Stein loss. Furthermore, we reformulate the
limiting result as
\begin{equation*}
\rho _{n}=\frac{3}{2}\sigma _{0}/n\Leftrightarrow \log (\rho _{n})=-\log
(n)+\log \left( \frac{3}{2}\sigma _{0}\right) \text{.}
\end{equation*}%
We then perform a regression on $\log (\rho _{n})$ with respect to $\log (n)
$. Figure \ref{rho} gives the relationship between $\rho $ and $n$ and the
regression line. We can find that $n\rho _{n}$ is approximately equal to a
constant, which is validated by the top right plot. The plots on the left
show the qualitative behavior of $\rho _{n}$; the figure on the top left
shows a behavior consistent with a decrease of order $O\left( 1/n\right) $,
the bottom left plot shows that $n^{1/2}\rho _{n}$ still decreases to zero,
indicating that $\rho _{n}$ converges to zero faster than the square-root
rate. The regression statistics, corresponding to the regression plot shown
in the bottom right of the plot, are shown in Table \ref{tab} and $%
R^{2}=0.97 $.

The theoretical constant $\log \left( 1.5\cdot \sigma _{0}\right) =1.5568$
is very close to the empirical regression intercept $1.5525,$ while the
coefficient multiplying $-\log \left( n\right) $ is close to unity. Hence,
the empirical result matches perfectly with our theory.


\begin{figure}[!h]
\centering
\includegraphics[width=8cm]{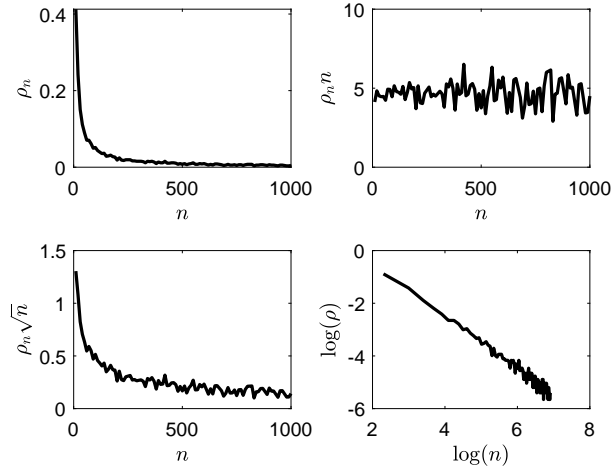}
\caption{$\protect\rho _{n}$ VS $n$ for 1-dimension normal distribution}
\label{rho}
\end{figure}

\begin{table}[!h]
\centering
{\small
\begin{tabular}{lcc}
\toprule & $\log(n)$ & constant \\
\midrule Coefficient & -1.0037 & 1.5525 \\
95\% Confidence interval & [-1.0387,-0.9687] & [1.3419,1.7631] \\
\bottomrule &  &
\end{tabular}%
}
\caption{Regression results for 1-dimension normal distribution}
\label{tab}
\end{table}

We provide additional examples involving higher dimensions. In the
subsequent examples, the data is sampled from a normal distribution $%
N(0,\Sigma _{0})$, where $\left( \Sigma _{0}\right) _{ij}=10\times
0.5^{|i-j|},$ $1\leq i,j\leq d$. We test the cases corresponding to $d=3$
and $d=5$ in the experiments. Due to computational constraints, we vary the
number of data points, $n$, ranging from 20 to 400. For each $n$, we use $%
T=100$ trials to compute empirically the optimal choice of uncertainty to
minimize the empirical Stein loss. Figures \ref{rho3} and \ref{rho5} show
the results for the 3-dimension and 5-dimension cases, respectively. Tables %
\ref{tab2} and \ref{tab3} give the regression statistics and $R^{2}=0.97$ in
both cases, and the performance is completely analogous to the one
dimensional case, thus empirically validating our theoretical results.

\begin{figure}[!h]
\centering
\includegraphics[width=8cm]{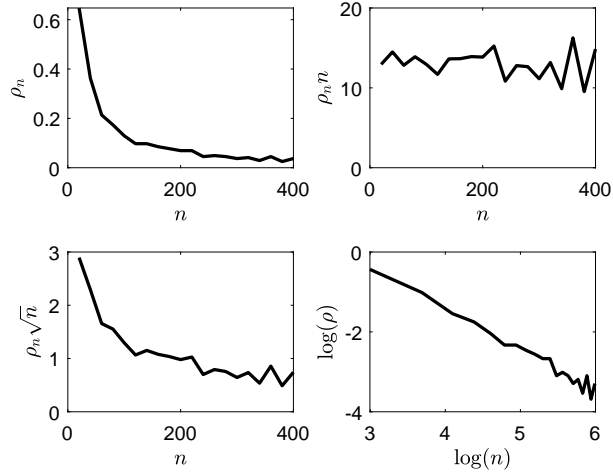}
\caption{$\protect\rho _{n}$ VS $n$ for 3-dimension normal distribution}
\label{rho3}
\end{figure}
\begin{table}[!h]
\centering
{\small
\begin{tabular}{lcc}
\toprule & $\log(n)$ & constant \\
\midrule Coefficient & -1.0340 & 2.7305 \\
95\% Confidence interval & [-1.1163,-0.9516] & [2.3045, 3.1565] \\
\bottomrule &  &
\end{tabular}%
}
\caption{Regression results for 3-dimension normal distribution}
\label{tab2}
\end{table}
\begin{figure}[!h]
\centering
\includegraphics[width=8cm]{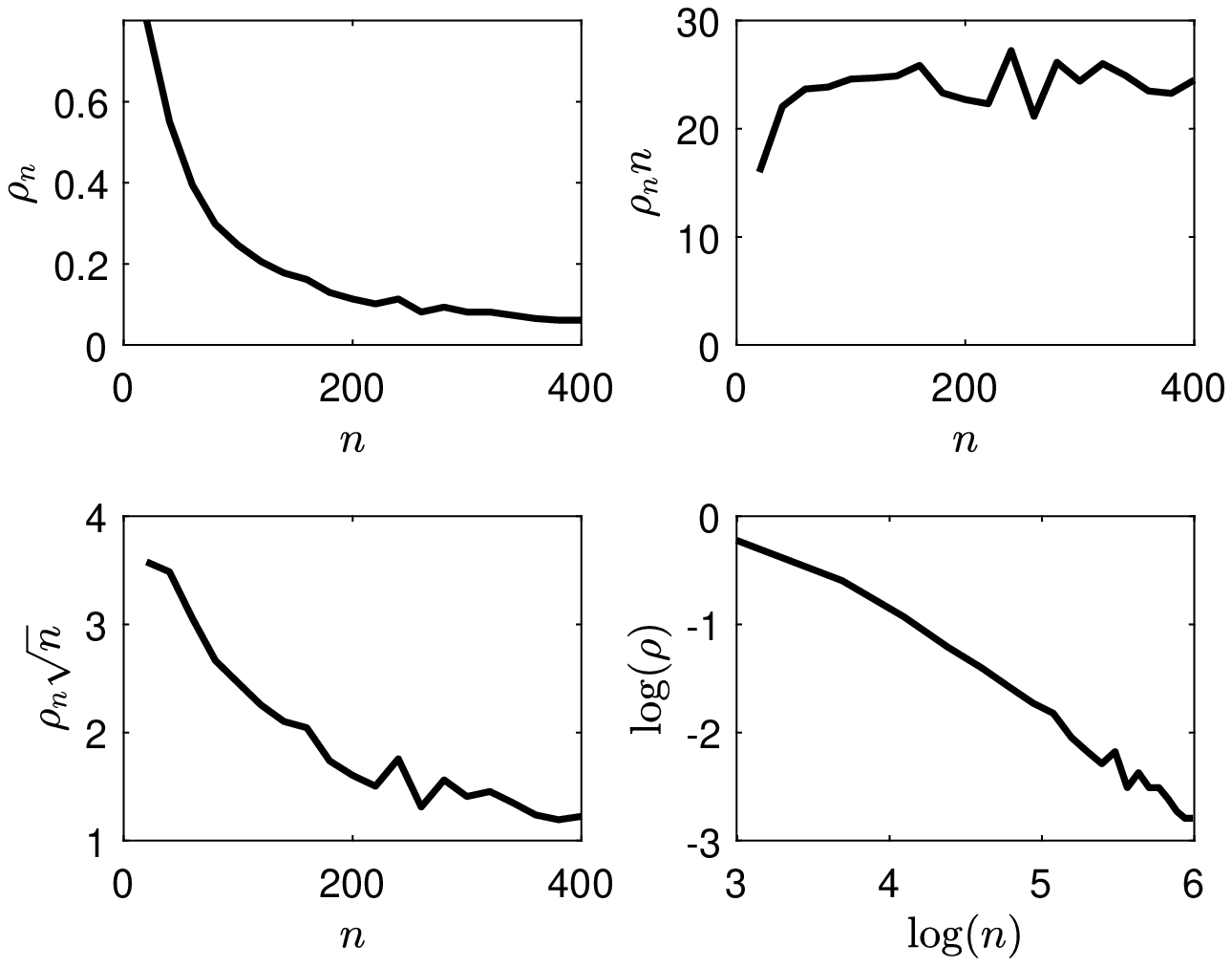}
\caption{$\protect\rho _{n}$ VS $n$ for 5-dimension normal distribution}
\label{rho5}
\end{figure}

\begin{table}[!h]
\centering
{\small
\begin{tabular}{lcc}
\toprule & $\log(n)$ & constant \\
\midrule Coefficient & -0.9177 & 2.7413 \\
95\% Confidence interval & [-0.9716 ,-0.8638] & [2.4625, 3.0201] \\
\bottomrule &  &
\end{tabular}%
}
\caption{Regression results for 5-dimension normal distribution}
\label{tab3}
\end{table}

\section*{Acknowledgements}

We gratefully acknowledge support from the following NSF grants 1915967, 1820942, 1838676 as well as DARPA award N660011824028.


\bibliographystyle{plain}
\bibliography{mybib}



\newpage \setcounter{section}{0} \setcounter{subsection}{0} %
\setcounter{equation}{0}
\renewcommand
{\thesection}{Appendix \Alph{section}} \renewcommand\thesubsection{%
\thesection.\arabic{subsection}} \renewcommand{\theequation}{\Alph{section}.%
\arabic{equation}}

\section{Proofs of Auxiliary Results}

\subsection{Proof of Lemma \protect\ref{lemma}}

We first restate a theorem in \cite{nguyen2018distributionally}.

\begin{theorem}[Theorem 3.1 in \protect\cite{nguyen2018distributionally}]
If $\rho >0$ and $\hat{\Sigma}_{n}$ admits the spectral decomposition $\hat{%
\Sigma}_{n}=\sum_{i=1}^{d}\hat{\lambda}_{i}\hat{v}_{i}\left( \hat{v}%
_{i}\right) ^{T}$ with eigenvalues $\hat{\lambda}_{i}$ and corresponding
orthonormal eigenvectors $\hat{v}_{i},$ $i\leq d,$ then the unique minimizer
of (\ref{DRO_sol}) is given by $X_{n}^{\ast }(\rho
)=\sum_{i=1}^{d}x_{i}^{\ast }\hat{v}_{i}\left( \hat{v}_{i}\right) ^{T},$
where
\begin{equation}
\hat{x}_{i}^{\ast }=\gamma ^{\ast }\left[ 1-\frac{1}{2}\left( \sqrt{\hat{%
\lambda}_{i}^{2}\left( \gamma ^{\ast }\right) ^{2}+4\hat{\lambda}_{i}\gamma
^{\ast }}-\hat{\lambda}_{i}\gamma ^{\ast }\right) \right] ,  \label{x_eqn}
\end{equation}%
and $\gamma ^{\ast }>0$ is the unique positive solution of the algebraic
equation
\begin{equation}
\left( \rho ^{2}-\frac{1}{2}\sum_{i=1}^{d}\hat{\lambda}_{i}\right) \gamma -d+%
\frac{1}{2}\sum_{i=1}^{d}\sqrt{\hat{\lambda}_{i}^{2}\gamma ^{2}+4\hat{\lambda%
}_{i}\gamma }=0.  \label{gamma_eqn}
\end{equation}
\end{theorem}

\begin{proof}[Proof of Lemma \protect\ref{lemma}.]
Since the underlying covariance matrix is invertible with probability one
when $n>d$, we have $\hat{\lambda}_{i}>0$ for $i=1,2,\ldots ,d$. We consider
the case $\rho \leq1$. Note that we have the following inequality,%
\begin{equation}
\sqrt{\hat{\lambda}_{i}^{2}\gamma ^{2}+4\hat{\lambda}_{i}\gamma }-\left(
\hat{\lambda}_{i}\gamma +2\right) =-\frac{4}{\sqrt{\hat{\lambda}%
_{i}^{2}\gamma ^{2}+4\hat{\lambda}_{i}\gamma }+\left( \hat{\lambda}%
_{i}\gamma +2\right) }\geq -\frac{2}{\hat{\lambda}_{i}\gamma }.
\label{first_inequ}
\end{equation}%
Then, (\ref{gamma_eqn}) gives us $\rho \gamma ^{\ast }\leq \sqrt{%
\sum_{i=1}^{d}\hat{\lambda}_{i}^{-1}}.$ On the other hand, we have
\begin{eqnarray}
\rho ^{2}\gamma ^{\ast } &=&\frac{1}{2}\sum_{i=1}^{d}\left( 2+\hat{\lambda}%
_{i}\gamma ^{\ast }-\sqrt{\hat{\lambda}_{i}^{2}\left( \gamma ^{\ast }\right)
^{2}+4\hat{\lambda}_{i}\gamma ^{\ast }}\right)  \notag \\
&=&\sum_{i=1}^{d}\frac{2}{\sqrt{\hat{\lambda}_{i}^{2}\left( \gamma ^{\ast
}\right) ^{2}+4\hat{\lambda}_{i}\left( \gamma ^{\ast }\right) ^{2}}+\left(
\hat{\lambda}_{i}\gamma ^{\ast }+2\right) }  \notag \\
&\geq &\sum_{i=1}^{d}\frac{1}{\hat{\lambda}_{i}\gamma ^{\ast }+2}
\label{large_gamma_inequ} \\
&\geq &d\frac{1}{\left( \max_{i}\hat{\lambda}_{i}\right) \gamma ^{\ast }+2}.
\notag
\end{eqnarray}%
Then, a basic property of the quadratic equation gives us that
\begin{equation*}
\gamma ^{\ast }\geq \frac{\sqrt{1+\left( \max_{i}\hat{\lambda}_{i}\right)
d/\rho ^{2}}-1}{\left( \max_{i}\hat{\lambda}_{i}\right) }\geq \frac{1}{4}%
\min \left\{ d/\rho ^{2},\frac{\sqrt{d}/\rho }{\sqrt{\max_{i}\hat{\lambda}%
_{i}}}\right\} .
\end{equation*}%
Furthermore, (\ref{large_gamma_inequ}) also shows that
\begin{eqnarray}
\rho ^{2}\gamma ^{\ast } &\geq &\left( \sum_{i=1}^{d}\frac{1}{\hat{\lambda}%
_{i}\gamma ^{\ast }}\right) \frac{1}{1+2/\left( \left( \min_{i}\hat{\lambda}%
_{i}\right) \gamma ^{\ast }\right) }  \notag \\
&\geq &\frac{1}{\gamma ^{\ast }}\left( \sum_{i=1}^{d}\hat{\lambda}%
_{i}^{-1}\right) \frac{\min \left\{ d/\rho ^{2},\frac{\sqrt{d}/\rho }{\sqrt{%
\max_{i}\hat{\lambda}_{i}}}\right\} }{\min \left\{ d/\rho ^{2},\frac{\sqrt{d}%
/\rho }{\sqrt{\max_{i}\hat{\lambda}_{i}}}\right\} +8/\left( \min_{i}\hat{%
\lambda}_{i}\right) }.  \label{gamma_bdd}
\end{eqnarray}%
By combining all of the above and noticing that $1+x\geq \sqrt{1+x}$ for $%
x\geq 0$, we have for $\rho \leq 1$
\begin{equation}
\frac{\rho }{\sqrt{\sum_{i=1}^{d}\hat{\lambda}_{i}^{-1}}}+\frac{\hat{M}_{n}}{%
\sqrt{\sum_{i=1}^{d}\hat{\lambda}_{i}^{-1}}}\rho ^{2}\geq \frac{1}{\gamma
^{\ast }}\geq \frac{1}{\sqrt{\sum_{i=1}^{d}\hat{\lambda}_{i}^{-1}}}\rho ,
\label{gamma}
\end{equation}%
where
\begin{equation*}
\hat{M}_{n}=\frac{8/\left( \min_{i}\hat{\lambda}_{i}\right) }{\min \left\{ d,%
\frac{\sqrt{d}}{\sqrt{\max_{i}\hat{\lambda}_{i}}}\right\} }.
\end{equation*}%
By plugging it to (\ref{gamma_eqn}), we have
\begin{equation*}
\frac{1}{x_{i}^{\ast }}=  \frac{\sqrt{\hat{\lambda}_{i}^{2}\left(
\gamma ^{\ast }\right) ^{2}+4\hat{\lambda}_{i}\gamma ^{\ast }}+\hat{\lambda}%
_{i}\gamma ^{\ast }+2}{2\gamma ^{\ast }}\leq \hat{\lambda}_{i}+2/\gamma ^{\ast }\leq \hat{%
\lambda}_{i}+\frac{2\rho }{\sqrt{\sum_{i=1}^{d}\hat{\lambda}_{i}^{-1}}}+%
\frac{2\hat{M}_{n}}{\sqrt{\sum_{i=1}^{d}\hat{\lambda}_{i}^{-1}}}\rho ^{2}.
\end{equation*}%
For the lower bound of $1/x_{i}^{\ast }$, we have%
\begin{equation*}
1/x_{i}^{\ast }=\frac{\sqrt{\hat{\lambda}_{i}^{2}\left( \gamma ^{\ast
}\right) ^{2}+4\hat{\lambda}_{i}\gamma ^{\ast }}+\hat{\lambda}_{i}\gamma
^{\ast }+2}{2\gamma ^{\ast }}\geq \frac{\sqrt{\hat{\lambda}_{i}^{2}\left(
\gamma ^{\ast }\right) ^{2}+4\hat{\lambda}_{i}\gamma ^{\ast }}}{\gamma
^{\ast }}.
\end{equation*}%
Then by (\ref{first_inequ}), we have for $\rho \leq 1$
\begin{eqnarray*}
\frac{\sqrt{\hat{\lambda}_{i}^{2}\left( \gamma ^{\ast }\right) ^{2}+4\hat{%
\lambda}_{i}\gamma ^{\ast }}}{\gamma ^{\ast }} &\geq &\hat{\lambda}%
_{i}+2/\gamma ^{\ast }-\frac{2}{\hat{\lambda}_{i}\left( \gamma ^{\ast
}\right) ^{2}} \\
&\geq &\hat{\lambda}_{i}+\frac{2\rho }{\sqrt{\sum_{i=1}^{d}\hat{\lambda}%
_{i}^{-1}}}-\frac{2\left( 1+\hat{M}_{n}\right) ^{2}}{\hat{\lambda}_{i}\left(
\sum_{i=1}^{d}\hat{\lambda}_{i}^{-1}\right) }\rho ^{2}
\end{eqnarray*}%
Therefore, we conclude that
\begin{equation*}
\frac{1}{x_{i}^{\ast }}=\hat{\lambda}_{i}+\frac{2\rho }{\sqrt{\sum_{i=1}^{d}%
\hat{\lambda}_{i}^{-1}}}+O(\rho ^{2}).
\end{equation*}%
and
\begin{equation*}
X_{n}^{\ast }(\rho )^{-1}=\hat{\Sigma}_{n}-\hat{\Sigma}_{n}\hat{A}_{n}\hat{%
\Sigma}_{n}\rho +O(\rho ^{2}),
\end{equation*}%
where
\begin{equation*}
\hat{A}_{n}=-\sum_{i=1}^{d}\frac{2\hat{v}_{i}\left( \hat{v}_{i}\right) ^{T}}{%
\hat{\lambda}_{i}^{2}\sqrt{\sum_{i=1}^{d}\hat{\lambda}_{i}^{-1}}}=-\frac{2}{%
\sqrt{\text{tr}(\hat{\Sigma}_{n}^{-1})}}\hat{\Sigma}_{n}^{-2}.
\end{equation*}%
Specifically, the remainder terms satisfy
\begin{equation*}
-\left( \frac{2\left( 1+\hat{M}_{n}\right) ^{2}}{\sum_{i=1}^{d}\hat{\lambda}%
_{i}^{-1}}\sum_{i=1}^{d}\frac{\hat{v}_{i}\left( \hat{v}_{i}\right) ^{T}}{%
\hat{\lambda}_{i}}\right) \rho ^{2}\preceq X_{n}^{\ast }(\rho )^{-1}-\hat{%
\Sigma}_{n}+\hat{\Sigma}_{n}\hat{A}_{n}\hat{\Sigma}_{n}\rho \preceq \left(
\frac{2\hat{M}_{n}}{\sqrt{\sum_{i=1}^{d}\hat{\lambda}_{i}^{-1}}}%
\sum_{i=1}^{d}\hat{v}_{i}\left( \hat{v}_{i}\right) ^{T}\right) \rho ^{2}.
\end{equation*}%
We complete the proof of (\ref{remainder2}).

For the the proof of (\ref{remainder}), note that (\ref{gamma_eqn})
indicates
\begin{eqnarray*}
\frac{d\gamma ^{\ast }}{d\rho } &=&\frac{-2\rho \gamma ^{\ast }}{\left( \rho
^{2}-\frac{1}{2}\sum_{i=1}^{d}\hat{\lambda}_{i}\right) +\frac{1}{2}%
\sum_{i=1}^{d}\hat{\lambda}_{i}\frac{\hat{\lambda}_{i}\gamma ^{\ast }+2}{%
\sqrt{\hat{\lambda}_{i}^{2}\left( \gamma ^{\ast }\right) ^{2}+4\hat{\lambda}%
_{i}\gamma ^{\ast }}}} \\
&=&-\frac{2\rho \gamma ^{\ast }}{\rho ^{2}+\sum_{i=1}^{d}\frac{2\hat{\lambda}%
_{i}}{\sqrt{\hat{\lambda}_{i}^{2}\left( \gamma ^{\ast }\right) ^{2}+4\hat{%
\lambda}_{i}\gamma ^{\ast }}\left( \hat{\lambda}_{i}\gamma ^{\ast }+2+\sqrt{%
\hat{\lambda}_{i}^{2}\left( \gamma ^{\ast }\right) ^{2}+4\hat{\lambda}%
_{i}\gamma ^{\ast }}\right) }}.
\end{eqnarray*}%
Since $\hat{\lambda}_{i}\gamma ^{\ast }\leq \sqrt{\hat{\lambda}%
_{i}^{2}\left( \gamma ^{\ast }\right) ^{2}+4\hat{\lambda}_{i}\gamma ^{\ast }}%
\leq \hat{\lambda}_{i}\gamma ^{\ast }+2$ and $\rho \gamma ^{\ast }\leq \sqrt{%
\sum_{i=1}^{d}\hat{\lambda}_{i}^{-1}},$ we have
\begin{equation*}
-\frac{2\rho \gamma ^{\ast }}{\rho ^{2}+\sum_{i=1}^{d}\frac{\hat{\lambda}_{i}%
}{\left( \hat{\lambda}_{i}\gamma ^{\ast }+2\right) ^{2}}}\leq \frac{\partial
\gamma ^{\ast }}{\partial \rho }\leq -\frac{2\rho \gamma ^{\ast }}{\rho
^{2}+\sum_{i=1}^{d}\frac{\hat{\lambda}_{i}}{\left( \hat{\lambda}_{i}\gamma
^{\ast }\right) ^{2}}}\leq -\frac{\rho \left( \gamma ^{\ast }\right) ^{3}}{%
\sum_{i=1}^{d}\hat{\lambda}_{i}^{-1}}.
\end{equation*}%
Then, by using the bound (\ref{gamma}), we further have
\begin{equation*}
-\frac{\rho \left( \gamma ^{\ast }\right) ^{3}}{\sum_{i=1}^{d}\hat{\lambda}%
_{i}^{-1}}\leq -\frac{\left( \gamma ^{\ast }\right) ^{2}}{\sqrt{%
\sum_{i=1}^{d}\hat{\lambda}_{i}^{-1}}}\left( 1-\hat{M}_{n}\rho \right) .
\end{equation*}%
Furthermore, the proof of Proposition 3.5 in \cite%
{nguyen2018distributionally} indicates that%
\begin{equation*}
\frac{\partial x_{i}}{\partial \gamma ^{\ast }}=1+\hat{\lambda}_{i}\gamma
^{\ast }-\frac{\hat{\lambda}_{i}^{2}\left( \gamma ^{\ast }\right) ^{2}+3\hat{%
\lambda}_{i}\gamma ^{\ast }}{\sqrt{\hat{\lambda}_{i}^{2}\left( \gamma ^{\ast
}\right) ^{2}+4\hat{\lambda}_{i}\gamma ^{\ast }}}.
\end{equation*}%
Let $z_{i}=\hat{\lambda}_{i}\gamma ^{\ast }$ for $i=1,2,\ldots ,d.$ We have%
\begin{equation*}
\frac{\partial x_{i}}{\partial \gamma ^{\ast }}=\frac{4z_{i}}{\sqrt{%
z_{i}^{2}+4z_{i}}\left( \left( 1+z_{i}\right) \sqrt{z_{i}^{2}+4z_{i}}%
+z_{i}^{2}+3z_{i}\right) }\in \left[ \frac{2z_{i}}{\left( z_{i}+2\right) ^{3}%
},\frac{2}{z_{i}^{2}}\right] .
\end{equation*}%
From (\ref{gamma_bdd}), we have%
\begin{equation}
\sum_{i=1}^{d}\frac{\hat{\lambda}_{i}}{\left( \hat{\lambda}_{i}\gamma ^{\ast
}+2\right) ^{2}}\geq \frac{1}{\left( \gamma ^{\ast }\right) ^{2}}\left(
\sum_{i=1}^{d}\hat{\lambda}_{i}^{-1}\right) \left( \frac{\min \left\{ d/\rho
^{2},\frac{\sqrt{d}/\rho }{\sqrt{\max_{i}\hat{\lambda}_{i}}}\right\} }{\min
\left\{ d/\rho ^{2},\frac{\sqrt{d}/\rho }{\sqrt{\max_{i}\hat{\lambda}_{i}}}%
\right\} +8/\left( \min_{i}\hat{\lambda}_{i}\right) }\right) ^{2}
\label{bdd_2_2}
\end{equation}%
and
\begin{equation*}
\frac{2z_{i}}{\left( z_{i}+2\right) ^{3}}\geq \frac{2}{\left( \hat{\lambda}%
_{i}\gamma ^{\ast }\right) ^{2}}\left( \frac{\min \left\{ d/\rho ^{2},\frac{%
\sqrt{d}/\rho }{\sqrt{\max_{i}\hat{\lambda}_{i}}}\right\} }{\min \left\{
d/\rho ^{2},\frac{\sqrt{d}/\rho }{\sqrt{\max_{i}\hat{\lambda}_{i}}}\right\}
+8/\left( \min_{i}\hat{\lambda}_{i}\right) }\right) ^{3}.
\end{equation*}%
Therefore, by combining (\ref{gamma_bdd}) and (\ref{bdd_2_2}), we have for $%
\rho \leq 1,$
\begin{eqnarray*}
\frac{\partial \gamma ^{\ast }}{\partial \rho } &\geq &-\frac{2\rho \left(
\gamma ^{\ast }\right) ^{3}}{\left( \sum_{i=1}^{d}\hat{\lambda}%
_{i}^{-1}\right) \left( \frac{\min \left\{ d/\rho ^{2},\frac{\sqrt{d}/\rho }{%
\sqrt{\max_{i}\hat{\lambda}_{i}}}\right\} }{\min \left\{ d/\rho ^{2},\frac{%
\sqrt{d}/\rho }{\sqrt{\max_{i}\hat{\lambda}_{i}}}\right\} +8/\left( \min_{i}%
\hat{\lambda}_{i}\right) }+\left( \frac{\min \left\{ d/\rho ^{2},\frac{\sqrt{%
d}/\rho }{\sqrt{\max_{i}\hat{\lambda}_{i}}}\right\} }{\min \left\{ d/\rho
^{2},\frac{\sqrt{d}/\rho }{\sqrt{\max_{i}\hat{\lambda}_{i}}}\right\}
+8/\left( \min_{i}\hat{\lambda}_{i}\right) }\right) ^{2}\right) } \\
&\geq &-\frac{\rho \left( \gamma ^{\ast }\right) ^{3}}{\left( \sum_{i=1}^{d}%
\hat{\lambda}_{i}^{-1}\right) }\left( 1+\hat{M}_{n}\rho \right) ^{2} \\
&\geq &-\frac{\rho \left( \gamma ^{\ast }\right) ^{3}}{\left( \sum_{i=1}^{d}%
\hat{\lambda}_{i}^{-1}\right) }\left( 1+\left( 2\hat{M}_{n}+\hat{M}%
_{n}^{2}\right) \rho \right) .
\end{eqnarray*}%
Similarly, we have for $\rho \leq 1,$%
\begin{equation*}
\frac{2z_{i}}{\left( z_{i}+2\right) ^{3}}\geq \frac{2}{\left( \hat{\lambda}%
_{i}\gamma ^{\ast }\right) ^{2}}\left( 1-\hat{M}_{n}\rho \right) ^{3}\geq
\frac{2}{\left( \hat{\lambda}_{i}\gamma ^{\ast }\right) ^{2}}\left( 1-\left(
\hat{M}_{n}^{3}+3\hat{M}_{n}\right) \rho \right) .
\end{equation*}%
Finally,$\ $by combining all of the above together and the chain rule, we
have%
\begin{equation*}
-\frac{2}{\hat{\lambda}_{i}^{2}}\frac{\rho \gamma ^{\ast }}{\left(
\sum_{i=1}^{d}\hat{\lambda}_{i}^{-1}\right) }\left( 1+\left( 2\hat{M}_{n}+%
\hat{M}_{n}^{2}\right) \rho \right) \leq \frac{\partial x_{i}}{\partial \rho
}\leq -\frac{2}{\hat{\lambda}_{i}^{2}\sqrt{\sum_{i=1}^{d}\hat{\lambda}%
_{i}^{-1}}}\left( 1-\hat{M}_{n}\rho \right) \left( 1-\left( \hat{M}_{n}^{3}+3%
\hat{M}_{n}\right) \rho \right) .
\end{equation*}%
After simplification, we have%
\begin{equation*}
-\frac{\left( 4\hat{M}_{n}+2\hat{M}_{n}^{2}\right) }{\hat{\lambda}_{i}^{2}%
\sqrt{\sum_{i=1}^{d}\hat{\lambda}_{i}^{-1}}}\rho \leq \frac{\partial x_{i}}{%
\partial \rho }+\frac{2}{\hat{\lambda}_{i}^{2}\sqrt{\sum_{i=1}^{d}\hat{%
\lambda}_{i}^{-1}}}\leq \frac{\left( 2\hat{M}_{n}^{3}+8\hat{M}_{n}\right) }{%
\hat{\lambda}_{i}^{2}\sqrt{\sum_{i=1}^{d}\hat{\lambda}_{i}^{-1}}}\rho .
\end{equation*}%
Furthermore, we have
\begin{equation*}
-\left( \frac{4\hat{M}_{n}+2\hat{M}_{n}^{2}}{\sqrt{\sum_{i=1}^{d}\hat{\lambda%
}_{i}^{-1}}}\sum_{i=1}^{d}\frac{\hat{v}_{i}\left( \hat{v}_{i}\right) ^{T}}{%
\hat{\lambda}_{i}^{2}}\right) \rho \preceq \frac{\partial X_{n}^{\ast }(\rho
)}{\partial \rho }-\hat{A}_{n}\preceq \left( \frac{2\hat{M}_{n}^{3}+8\hat{M}%
_{n}}{\sqrt{\sum_{i=1}^{d}\hat{\lambda}_{i}^{-1}}}\sum_{i=1}^{d}\frac{\hat{v}%
_{i}\left( \hat{v}_{i}\right) ^{T}}{\hat{\lambda}_{i}^{2}}\right) \rho .
\end{equation*}%
This completes the proof.
\end{proof}

\subsection{Proof of Proposition \protect\ref{probCLT}}

{\color{blue} }

\begin{proof}[Proof of (1).]
The proof follows from the standard central limit theorem (CLT).
\end{proof}

\begin{proof}[Proof of (2).]
Since $\hat{\Sigma}_{n}$ is the average of i.i.d copies $\xi _{i}\xi
_{i}^{T},$ the result follows by CLT.
\end{proof}

\begin{proof}[Proof of (3).]
 The first statement follows from the continuous mapping
theorem and $\hat{\Sigma}_{n}\overset{p}{\rightarrow }\Sigma _{0}.$ Let $%
f(\Sigma )=-2\left( \text{tr}(\Sigma ^{-1})\right) ^{-1/2}\Sigma ^{-2}$,
where $\Sigma $ is positive-definite matrix. We now expand $f\left( \Sigma
+hA\right) $ for any matrix $A$ as the scalar $h>0$ tends to zero to obtain
a representation for the gradient of $f(\Sigma )$, $Df(\Sigma )$. This
expansion yields%
\begin{eqnarray*}
f\left( \Sigma +hA\right) &=&-2\left( \text{tr}(\left( \Sigma +hA\right)
^{-1})\right) ^{-1/2}\left( \Sigma +hA\right) ^{-2} \\
&=&-2\left( \text{tr}(\Sigma ^{-1})-\text{tr}\left( h\Sigma ^{-1}A\Sigma
^{-1}\right) +o\left( h\right) \right) ^{-1/2}\left( \left( I+h\Sigma
^{-1}A\right) ^{-1}\Sigma ^{-1}\right) ^{2} \\
&=&-2\text{tr}(\Sigma ^{-1})^{-1/2}\left( 1-h\frac{\text{tr}\left( \Sigma
^{-1}A\Sigma ^{-1}\right) }{\text{tr}(\Sigma ^{-1})}\right) ^{-1/2}\left(
\Sigma ^{-2}-h\Sigma ^{-1}A\Sigma ^{-2}-h\Sigma ^{-2}A\Sigma ^{-1}\right)
+o\left( h\right) \\
&=&-2\text{tr}(\Sigma ^{-1})^{-1/2}\left( 1+h\frac{\text{tr}\left( \Sigma
^{-1}A\Sigma ^{-1}\right) }{2\text{tr}(\Sigma ^{-1})}\right) \left( \Sigma
^{-2}-h\Sigma ^{-1}A\Sigma ^{-2}-h\Sigma ^{-2}A\Sigma ^{-1}\right) +o\left(
h\right) \\
&=&-2\text{tr}(\Sigma ^{-1})^{-1/2}\left( \Sigma ^{-2}+h\frac{\text{tr}%
\left( \Sigma ^{-1}A\Sigma ^{-1}\right) \Sigma ^{-2}}{2\text{tr}(\Sigma
^{-1})}-h\Sigma ^{-1}A\Sigma ^{-2}-h\Sigma ^{-2}A\Sigma ^{-1}\right)
+o\left( h\right) \\
&=&f\left( \Sigma \right) -h\frac{\text{tr}\left( \Sigma ^{-1}A\Sigma
^{-1}\right) \Sigma ^{-2}}{\text{tr}(\Sigma ^{-1})^{3/2}}+2h\frac{\Sigma
^{-1}A\Sigma ^{-2}+\Sigma ^{-2}A\Sigma ^{-1}}{\text{tr}(\Sigma ^{-1})^{1/2}}%
+o\left( h\right) ,
\end{eqnarray*}%
which, in turn, results in the linear operator satisfying for any $A\in
\mathbb{R}^{d\times d}$%
\begin{equation}
Df(\Sigma )A=-\frac{\text{tr}\left( \Sigma ^{-1}A\Sigma ^{-1}\right) \Sigma
^{-2}}{\text{tr}(\Sigma ^{-1})^{3/2}}+2\frac{\Sigma ^{-1}A\Sigma
^{-2}+\Sigma ^{-2}A\Sigma ^{-1}}{\text{tr}(\Sigma ^{-1})^{1/2}}.  \label{Df}
\end{equation}%
After applying the delta method, we have the desired result.
\end{proof}

\subsection{Proof of Proposition \protect\ref{large_zero}}

We first note the following elementary result, which is standard in matrix
algebra.

\begin{lemma}
\label{trace_ineqn}For any $d\times d$ matrices $A,B$ (real valued) we have%
\begin{equation*}
\text{tr}(A^{T}A)\text{tr}(B^{T}B)\geq \text{tr}(A^{T}B)^{2}=\left\vert
\left\langle A,B\right\rangle \right\vert ^{2},
\end{equation*}%
where strict inequality holds unless $A$ is a multiple of $B$.
\end{lemma}

\begin{proof}[Proof of Lemma \protect\ref{trace_ineqn}]
By the Cauchy-Schwarz inequality, we have%
\begin{equation*}
\text{tr}(A^{T}A)\text{tr}(B^{T}B)=\left( \sum_{i,j}A_{ij}^{2}\right) \left(
\sum_{i,j}B_{ij}^{2}\right) \geq \left( \sum_{i,j}A_{ij}B_{ij}\right)
^{2}=\left\vert \left\langle A,B\right\rangle \right\vert ^{2}.
\end{equation*}
\end{proof}

Now we proceed with the proof of Proposition \ref{large_zero}.

\begin{proof}[Proof of (1)]
It suffices to show that $\left\langle Z,Z_{A}\right\rangle \geq 0$ with
probability one and that $\left\langle Z,Z_{A}\right\rangle >0$ with
positive probability. Note that%
\begin{equation*}
\left\langle Z,Z_{A}\right\rangle =-\frac{\text{tr}(Z\Sigma _{0}^{-2})^{2}}{%
\text{tr}(\Sigma _{0}^{-1})^{3/2}}+\frac{4\text{tr}\left( \Sigma
_{0}^{-1}Z\Sigma _{0}^{-2}Z\right) }{\text{tr}(\Sigma _{0}^{-1})^{1/2}}.
\end{equation*}%
We will show that%
\begin{equation}
\text{tr}(\Sigma _{0}^{-1})\text{tr}\left( \Sigma _{0}^{-1}Z\Sigma
_{0}^{-2}Z\right) \geq \text{tr}(Z\Sigma _{0}^{-2})^{2}  \label{matrix_inequ}
\end{equation}%
follows from Lemma \ref{trace_ineqn}. This implies that $\left\langle
Z,Z_{A}\right\rangle \geq 0$. The equality holds if and only
if there exists $a\geq 0$ such that {$Z\Sigma _{0}^{-2}Z=aI,$ which is
equivalent to }$Z=\sqrt{a}\Sigma _{0}${. We know that }$Z\neq \sqrt{a}\Sigma
_{0}$ with probability one. Thus, $\left\langle Z,Z_{A}\right\rangle >0$
with probability one.

To show (\ref{matrix_inequ}), we use the Polar factorization (see, for
example, Chapter 4.2 in \cite{golub2012matrix}) for positive definite
matrices. That is, we write $\Sigma _{0}^{1/2}\Sigma _{0}^{1/2}=\Sigma _{0}$%
, where $\Sigma _{0}^{1/2}$ is a symmetric positive definite matrix. Note
that we can write
\begin{equation*}
Z=\Sigma _{0}^{1/2}W\Sigma _{0}^{1/2},
\end{equation*}%
where $W=\Sigma _{0}^{-1/2}Z\Sigma _{0}^{-1/2}$ is a symmetric matrix. To
recover the matrices $A$ and $B$, we let%
\begin{equation*}
A=\Sigma _{0}^{-1/2}\text{, }S=\Sigma _{0}\text{ and }B=WS^{-1/2}.
\end{equation*}%
Note that%
\begin{eqnarray*}
&&\text{tr}\left( \Sigma _{0}^{-1}Z\Sigma _{0}^{-2}Z\right) =\text{tr}\left(
Z\Sigma _{0}^{-1}\cdot \Sigma _{0}^{-1}Z\Sigma _{0}^{-1}\right) \\
&=&\text{tr}\left( S^{1/2}WS^{1/2}S^{-1/2}S^{-1/2}\cdot
S^{-1/2}S^{-1/2}\left( S^{1/2}WS^{1/2}\right) S^{-1/2}S^{-1/2}\right) \\
&=&\text{tr}\left( S^{1/2}WS^{-1/2}\cdot S^{-1/2}WS^{-1/2}\right) =\text{tr}%
\left( S^{-1/2}WWS^{-1/2}\right) =\text{tr}\left( B ^{T}B\right) .
\end{eqnarray*}

Therefore, this verifies that the choice of $B$ is consistent with the use
of Lemma \ref{trace_ineqn}. Clearly, $AA^{T}=\Sigma _{0}^{-1}$, thus making
this choice also consistent with Lemma \ref{trace_ineqn}. Finally, we have
that%
\begin{eqnarray*}
\text{tr}(Z\Sigma _{0}^{-2}) &=&\text{tr}(S^{1/2}WS^{1/2}S^{-1/2}S^{-1/2}%
\Sigma _{0}^{-1}) \\
&=&\text{tr}(S^{1/2}WS^{-1/2}S^{-1/2}S^{-1/2}) \\
&=&\text{tr}(WS^{-1/2}S^{-1/2})=\text{tr}(S^{-1/2}S^{-1/2}W)=\text{tr}\left(
A^{T}B\right) .
\end{eqnarray*}%
The result then follows.
\end{proof}

{\color{blue} }

\begin{proof}[Proof of (2):]
\begin{eqnarray}
\left\langle \hat{\Sigma}_{n}-\Sigma _{0},\hat{A}_{n}-A_{0}\right\rangle
&=&-2\left\langle \hat{\Sigma}_{n}-\Sigma _{0},\frac{1}{\sqrt{\text{tr}(\hat{%
\Sigma}_{n}^{-1})}}\hat{\Sigma}_{n}^{-2}-\frac{1}{\sqrt{\text{tr}(\Sigma
_{0}^{-1})}}\Sigma _{0}^{-2}\right\rangle  \label{2_inequ} \\
&=&2\left( -\sqrt{\text{tr}(\hat{\Sigma}_{n}^{-1})}-\sqrt{\text{tr}(\Sigma
_{0})}+\frac{\text{tr}(\hat{\Sigma}_{n}^{-1}\Sigma _{0}\hat{\Sigma}_{n}^{-1})%
}{\sqrt{\text{tr}(\hat{\Sigma}_{n}^{-1})}}+\frac{\text{tr}(\Sigma _{0}^{-1}%
\hat{\Sigma}_{n}\Sigma _{0}^{-1})}{\sqrt{\text{tr}(\Sigma _{0}^{-1})}}%
\right) .  \notag
\end{eqnarray}%
By Lemma (\ref{trace_ineqn}) and similar arguments with (1), we have
\begin{equation}
\text{tr}(\hat{\Sigma}_{n}^{-1}\Sigma _{0}\hat{\Sigma}_{n}^{-1})\geq \frac{%
\text{tr}(\hat{\Sigma}_{n}^{-1})^{2}}{\text{tr}(\Sigma _{0}^{-1})}\text{ and
tr}(\Sigma _{0}^{-1}\hat{\Sigma}_{n}\Sigma _{0}^{-1})\geq \frac{\text{tr}%
(\Sigma _{0}^{-1})^{2}}{\text{tr}(\hat{\Sigma}_{n}^{-1})}.
\label{norm_inequ}
\end{equation}%
By plugging (\ref{norm_inequ}) into (\ref{2_inequ}), we have%
\begin{equation*}
\left\langle \hat{\Sigma}_{n}-\Sigma _{0},\hat{A}_{n}-A_{0}\right\rangle
\geq 2\left( -\sqrt{\text{tr}(\hat{\Sigma}_{n}^{-1})}-\sqrt{\text{tr}(\Sigma
_{0})}+\frac{\text{tr}(\hat{\Sigma}_{n}^{-1})^{3/2}}{\text{tr}(\Sigma
_{0}^{-1})}+\frac{\text{tr}(\Sigma _{0}^{-1})^{3/2}}{\text{tr}(\hat{\Sigma}%
_{n}^{-1})}\right) .
\end{equation*}%
Consider the function $g:\mathbb{R}_{+}\times \mathbb{R}_{+}\rightarrow \mathbb{R%
}$,
\begin{equation*}
g(a,b)=-a-b+\frac{b^{3}}{a^{2}}+\frac{a^{3}}{b^{2}}=\frac{%
(a^{3}-b^{3})(a^{2}-b^{2})}{a^{2}b^{2}}\geq 0,
\end{equation*}%
and the equality holds if $a=b.$ Since $\sqrt{\text{tr}(\hat{\Sigma}%
_{n}^{-1})}=\sqrt{\text{tr}(\Sigma _{0})}$ with probability zero, the
desired result follows.
\end{proof}

\subsection{Proof of Lemma \protect\ref{temp}}

We first collect a few results from linear algebra (see, for example,
equation (2.3.3) and (2.3.7) in \cite{golub2012matrix}).

\begin{lemma}
\label{Lem_Aux_UI1}For any $d\times d$ matrix $A$ (real valued) we define $%
\left\Vert A\right\Vert _{F}^{2}=\left\langle A,A\right\rangle = \mathrm{tr}
\left( A^{T}A\right) $ (the Frobenius norm) and let $\left\Vert A\right\Vert
_{2}^{2}=\left\vert \lambda _{\max }\left( A^{T}A\right) \right\vert $
(where $\lambda _{\max }\left( B\right) $ is the eigenvalue of largest
modulus of the matrix $B$). Then, for any $A,B$ matrices of size $d\times d$
with real valued elements we have
\begin{equation*}
\left\Vert AB\right\Vert _{F}\leq \left\Vert A\right\Vert _{F}\left\Vert
B\right\Vert _{F},\text{ \ }\left\Vert B\right\Vert _{2}\leq \left\Vert
B\right\Vert _{F}.
\end{equation*}
\end{lemma}

In addition, we have the following properties of the distribution of $\hat{%
\Sigma}_{n}$, which follows the Wishart law (see, for example, Theorem
13.3.2 in \cite{anderson2003introduction}).

\begin{lemma}
\label{Lem_Aux_UI2}Assume $n>d.$ Let us write $\xi _{i}=C\zeta _{i}$ where $%
C\in \mathbb{R}^{d\times d}$ and $CC^{T}=\Sigma _{0}$ and put%
\begin{equation*}
S_{n}=C\left( \sum_{i=1}^{n}\zeta _{i}\zeta _{i}^{T}\right) C^{T}.
\end{equation*}%
Note that $\hat{\Sigma}_{n}=S_{n}/n$. Then, $S_{n}$ follows Wishart
distribution with parameters $d,n$ and $\Sigma _{0}$ (denoted as $%
W_{d}\left( n,\Sigma _{0}\right) $). Equivalently, $W=C^{-1}S_{n}\left(
C^{T}\right) ^{-1}$ is distributed $W_{d}\left( n,I\right) ,$ where $I$
denotes the $d\times d$ identity matrix. Moreover, the eigenvalue
distribution of $W$ satisfies
\begin{eqnarray*}
&&f_{w_{\left( 1\right) },...,w_{\left( d\right) }}\left(
w_{1},...,w_{d}\right) \\
&=&c_{d}\prod\limits_{i=1}^{d}\frac{\exp \left( -w_{i}/2\right) }{%
2^{n/2}\Gamma \left( \left( n-i+1\right) /2\right) }w_{i}^{\left(
n-d-1\right) /2}\prod\limits_{j>i}\left( w_{j}-w_{i}\right) \mathbb{I}\left(
0<w_{1}<...<w_{d}\right) .
\end{eqnarray*}%
where $\Gamma (\cdot )$ is the gamma function, $c_{d}$ is a constant
independent of $n$, and $\mathbb{I}(\cdot)$ is the indicator function.
\end{lemma}

We are now ready to provide the proof of Lemma \ref{temp}. By Proposition (%
\ref{probCLT}), Slutsky's theorem, and the continuous mapping theorem, we
have
\begin{eqnarray*}
&&\left\langle \hat{\Sigma}_{n}\hat{A}_{n}\hat{\Sigma}_{n},\hat{A}%
_{n}\right\rangle \overset{p}{\rightarrow }\left\langle \Sigma
_{0}A_{0}\Sigma _{0},A_{0}\right\rangle , \\
&&\left\langle \sqrt{n}\left( \hat{\Sigma}_{n}-\Sigma _{0}\right) ,\hat{A}%
_{n}\right\rangle \Rightarrow \left\langle Z,A_{0}\right\rangle , \\
&&\left\langle \sqrt{n}\left( \hat{\Sigma}_{n}-\Sigma _{0}\right) ,\sqrt{n}%
\left( \hat{A}_{n}-A_{0}\right) \right\rangle \Rightarrow \left\langle
Z,Z_{A}\right\rangle .
\end{eqnarray*}
Therefore, to verify Lemma \ref{temp}, we need to show the uniform
integrability of $\left\langle \hat{\Sigma}_{n}\hat{A}_{n}\hat{\Sigma}_{n},%
\hat{A}_{n}\right\rangle ,$ $\left\langle \sqrt{n}\left( \hat{\Sigma}%
_{n}-\Sigma _{0}\right) ,\hat{A}_{n}\right\rangle $ and $\left\langle \sqrt{n%
}\left( \hat{\Sigma}_{n}-\Sigma _{0}\right) ,\sqrt{n}\left( \hat{A}%
_{n}-A_{0}\right) \right\rangle .$ In turn, it suffices to verify that for
some $r>1$ and some $n_{0}<\infty $ we have
\begin{eqnarray*}
&&\sup_{n\geq n_{0}}\mathbb{E}\left[ \left\vert \left\langle \hat{\Sigma}_{n}%
\hat{A}_{n}\hat{\Sigma}_{n},\hat{A}_{n}\right\rangle \right\vert ^{r}\right]
<\infty , \\
&&\sup_{n\geq n_{0}}\mathbb{E}\left[ \left\vert \left\langle \sqrt{n}\left(
\hat{\Sigma}_{n}-\Sigma _{0}\right) ,\hat{A}_{n}\right\rangle \right\vert
^{r}\right] <\infty , \\
&&\sup_{n\geq n_{0}}\mathbb{E}\left[ \left\vert \left\langle \sqrt{n}\left(
\hat{\Sigma}_{n}-\Sigma _{0}\right) ,\sqrt{n}\left( \hat{A}_{n}-A_{0}\right)
\right\rangle \right\vert ^{r}\right] <\infty ,
\end{eqnarray*}%
(see, for example, Chapter 5 in \cite{durrett2010probability}).

\begin{proof}[Proof of (1). ]
From Lemma \ref{Lem_Aux_UI2} we have
\begin{eqnarray*}
&&f_{w_{\left( 1\right) },...,w_{\left( d\right) }}\left(
w_{1},...,w_{d}\right) \\
&=&c_{d}\prod\limits_{i=1}^{d}\frac{\exp \left( -w_{i}/2\right) }{%
2^{n/2}\Gamma \left( \left( n-i+1\right) /2\right) }w_{i}^{\left(
n-d-1\right) /2}\prod\limits_{j>i}\left( w_{j}-w_{i}\right) \mathbb{I}\left(
0<w_{1}<...<w_{d}\right) \\
&=&c_{d}\prod\limits_{i=1}^{d}\left( \frac{\exp \left( -w_{i}/2\right) }{%
2^{n/2}\Gamma \left( \left( n-i+1\right) /2\right) }w_{i}^{\left(
n-d-1\right) /2}w_{i}^{d-i}\prod\limits_{j>i}\left( w_{j}/w_{i}-1\right)
\mathbb{I}\left( 0<w_{1}<...<w_{d}\right) \right) \\
&=&c_{d}\prod\limits_{i=1}^{d}\left( \frac{\exp \left( -w_{i}/2\right) }{%
2^{n/2}\Gamma \left( \left( n-i+1\right) /2\right) }w_{i}^{\left(
n-i+1\right) /2-1}w_{i}^{{(d-i)/2}}\prod\limits_{j>i}\left(
w_{j}/w_{i}-1\right) \mathbb{I}\left( 0<w_{1}<...<w_{d}\right) \right) \\
&=&c_{d}^{\prime }\left( \prod\limits_{i=1}^{d}f_{\chi _{n-i+1}^{2}}\left(
w_{i}\right) \right) \prod\limits_{i=1}^{d}\left( \frac{w_{i}}{n}\right) ^{{%
(d-i)/2}}\prod\limits_{j>i}\left[ n^{1/2}\left( w_{j}/w_{i}-1\right) \right]
\mathbb{I}\left( 0<w_{1}<...<w_{d}\right) ,
\end{eqnarray*}%
where $f_{\chi _{n-i+1}^{2}}\left( \cdot \right) $ denotes the density of a
chi-squared distribution with $n-i+1$ degrees of freedom and $c_{d}^{\prime
} $ is another constant also independent of $n$. The previous identity can
be interpreted as follows. Let $W^{\left( n\right) }:=(W_{\left( 1\right)
}^{\left( n\right) },...,W_{\left( d\right) }^{\left( n\right) })$ be the
eigenvalues of a $W_{d}\left( n,I\right) $ random matrix, and let $\Lambda
\left( n\right) :=\left( \Lambda _{1}\left( n\right) ,...,\Lambda _{d}\left(
n\right) \right) $ be independent random variables such that $\Lambda
_{i}\left( n\right) \sim \chi _{n-i+1}^{2}$. Then for any positive (and
measurable) function $g:\mathbb{R}^{d}\rightarrow \lbrack 0,\infty )$, we
have
\begin{eqnarray}
&&\mathbb{E}\left[ g\left( W^{\left( n\right) }\right) \right]  \notag \\
&=&c_{d}^{\prime }\mathbb{E}\left[ g\left( \Lambda \left( n\right) \right)
\prod\limits_{i=1}^{d}\left( \frac{\Lambda _{i}\left( n\right) }{n}\right) ^{%
{(d-i)/2}}\prod\limits_{j>i}\left\vert n^{1/2}\left( \Lambda _{j}\left(
n\right) /\Lambda _{i}\left( n\right) -1\right) \right\vert \mathbb{I}\left(
0<\Lambda _{1}\left( n\right) <...<\Lambda _{d}\left( n\right) \right) %
\right]  \notag \\
&\leq &c_{d}^{\prime }\mathbb{E}\left[ g\left( \Lambda \left( n\right)
\right) \prod\limits_{i=1}^{d}\left( \frac{\Lambda _{i}\left( n\right) }{n}%
\right) ^{{(d-i)/2}}\prod\limits_{j>i}\left\vert n^{1/2}\left( \Lambda
_{j}\left( n\right) /\Lambda _{i}\left( n\right) -1\right) \right\vert %
\right]  \notag \\
&\leq &c_{d}^{\prime }\sqrt{\mathbb{E}\left[ g\left( \Lambda \left( n\right)
\right) ^{2}\right] }\sqrt{\mathbb{E}\left[ \left(
\prod\limits_{i=1}^{d}\left( \frac{\Lambda _{i}\left( n\right) }{n}\right) ^{%
{(d-i)/2}}\prod\limits_{j>i}\left\vert n^{1/2}\left( \Lambda _{j}\left(
n\right) /\Lambda _{i}\left( n\right) -1\right) \right\vert \right) ^{2}%
\right] }  \notag \\
&\leq &c_{d}^{\prime }\sqrt{\mathbb{E}\left[ g\left( \Lambda \left( n\right)
\right) ^{2}\right] }\mathbb{E}^{1/4}\left( \prod\limits_{i=1}^{d}\left(
\frac{\Lambda _{i}\left( n\right) }{n}\right) ^{2{(d-i)}}\right) \mathbb{E}%
^{1/4}\left( \prod\limits_{j>i}\left\vert n^{1/2}\left( \Lambda _{j}\left(
n\right) /\Lambda _{i}\left( n\right) -1\right) \right\vert ^{4}\right) ,
\label{Id_A}
\end{eqnarray}

where the last two inequalities are obtained by the Cauchy-Schwarz
inequality. We will show $\sup_{n\geq n_{0}}\mathbb{E}\left[ \left\vert
\left\langle \hat{\Sigma}_{n}\hat{A}_{n}\hat{\Sigma}_{n},\hat{A}%
_{n}\right\rangle \right\vert ^{r}\right] <\infty $ to verify the first
statement of Lemma \ref{temp}. Note that we can simplify $\left\vert
\left\langle \hat{\Sigma}_{n}\hat{A}_{n}\hat{\Sigma}_{n},\hat{A}%
_{n}\right\rangle \right\vert ^{r}$ as
\begin{equation*}
\left\vert \left\langle \hat{\Sigma}_{n}\hat{A}_{n}\hat{\Sigma}_{n},\hat{A}%
_{n}\right\rangle \right\vert ^{r}=\left\vert \text{tr}\left( \hat{A}_{n}%
\hat{\Sigma}_{n}\hat{A}_{n}\hat{\Sigma}_{n}\right) \right\vert ^{r}=\frac{%
4^{r}}{\text{tr}\left( \hat{\Sigma}_{n}^{-1}\right) ^{r}}\mathrm{tr}\left( \hat{\Sigma%
}_{n}^{-2}\right) ^{r}.
\end{equation*}%
By our definition, we have%
\begin{equation*}
\hat{\Sigma}_{n}^{-1}=\left( C^{T}\right) ^{-1}\left( W/n\right) ^{-1}C^{-1},
\end{equation*}%
and thus there exist numerical constants \underline{$c$}$_{1},\bar{c}_{1}>0$
such that
\begin{equation*}
\underline{c}_{1}\text{tr}\left( \left( W/n\right) ^{-1}\right) \leq \text{tr%
}\left( \hat{\Sigma}_{n}^{-1}\right) \leq \bar{c}_{1}\text{tr}\left( \left(
W/n\right) ^{-1}\right) .
\end{equation*}%
Similarly, there exist numerical constants \underline{$c$}$_{2},\bar{c}%
_{2}>0 $ such that%
\begin{equation*}
\underline{c}_{2}\text{tr}\left( \left( W/n\right) ^{-2}\right) \leq \text{tr%
}\left( \hat{\Sigma}_{n}^{-2}\right) \leq \bar{c}_{2}\text{tr}\left( \left(
W/n\right) ^{-2}\right) .
\end{equation*}%
After using the Cauchy-Schwarz inequality again, we have
\begin{eqnarray}
\mathbb{E}\left\vert \left\langle \hat{\Sigma}_{n}\hat{A}_{n}\hat{\Sigma}%
_{n},\hat{A}_{n}\right\rangle \right\vert ^{r} &\leq &c\mathbb{E}\left[
\sum_{i=1}^{d}\left( \frac{W_{\left( i\right) }^{\left( n\right) }}{n}%
\right) ^{r}\sum_{i=1}^{d}\left( \frac{n}{W_{\left( i\right) }^{\left(
n\right) }}\right) ^{2r}\right]  \notag \\
&\leq &c\sqrt{\mathbb{E}\left[ \sum_{i=1}^{d}\left( \frac{W_{\left( i\right)
}^{\left( n\right) }}{n}\right) ^{2r}\right] }\sqrt{\mathbb{E}\left[
\sum_{i=1}^{d}\left( \frac{n}{W_{\left( i\right) }^{\left( n\right) }}%
\right) ^{4r}\right] },  \label{g_func}
\end{eqnarray}%
where $c$ is a numerical constant$.$ Therefore, it suffices to show that for
any $r>1$ there exists $n_{0}$ such that%
\begin{equation}
\sup_{n\geq n_{0}}\left( \mathbb{E}\left[ \left( \frac{\Lambda _{j}\left(
n\right) }{n}\right) ^{r}\right] \cdot \mathbb{E}\left[ \left( \frac{n}{%
\Lambda _{i}\left( n\right) }\right) ^{r}\right] \cdot \mathbb{E}\left[
\left\vert n^{1/2}\left( \Lambda _{j}\left( n\right) /\Lambda _{i}\left(
n\right) -1\right) \right\vert ^{r}\right] \right) <\infty .  \label{AU0}
\end{equation}%
We know that $\Lambda _{i}\left( n\right) /n$ follows the gamma distribution
with shape parameter $\alpha =\left( n-i+1\right) /2$ and scale parameter $%
\lambda =n/2$. Write $Y_{n}\sim Gamma\left( \alpha ,\lambda \right) $ and
note that
\begin{eqnarray}
\mathbb{E}\left( \frac{1}{Y_{n}^{r}}\right) &=&\int_{0}^{\infty }\frac{1}{%
y^{r}}\frac{\exp \left( -\lambda y\right) \lambda ^{\alpha }y^{\alpha -1}}{%
\Gamma \left( \alpha \right) }\mathrm{d}y \notag \\
&=&\int_{0}^{\infty }\frac{\Gamma \left( \alpha -t\right) \lambda ^{r}\exp
\left( -\lambda y\right) \lambda ^{\alpha -r}y^{\alpha -r-1}}{\Gamma \left(
\alpha \right) \Gamma \left( \alpha -r\right) }\mathrm{d}y  \notag \\
&=&\frac{\Gamma \left( \alpha -r\right) \lambda ^{r}}{\Gamma \left( \alpha
\right) }.  \label{AU2}
\end{eqnarray}%
It follows from standard properties of the gamma function that lim$%
_{n\rightarrow \infty }\Gamma \left( \alpha -r\right) \lambda ^{r}/\Gamma
\left( \alpha \right) =1$ (see, for example, Chapter 3 in \cite%
{hogg1978introduction}). After applying exactly the same approach to $%
\mathbb{E}\left[ \left( \Lambda _{i}\left( n\right) /n\right) ^{r}\right] $,
we have
\begin{equation}
\mathbb{E}\left[ \left( \Lambda _{i}\left( n\right) /n\right) ^{r}\right] =%
\frac{\Gamma \left( \alpha +r\right) }{\Gamma \left( \alpha \right) \lambda
^{r}}\rightarrow 1,
\label{AU22}
\end{equation}%
as $n\rightarrow \infty .$

Now, we only need to show the third term in (\ref{AU0}) is finite. Note that
\begin{eqnarray}
&&\mathbb{E}\left[ \left\vert n^{1/2}\left( \Lambda _{j}\left( n\right)
/\Lambda _{i}\left( n\right) -1\right) \right\vert ^{r}\right]  \label{AU2C}
\\
&=&\mathbb{E}\left[ \left\vert n^{1/2}\left( \Lambda _{j}\left( n\right)
/\Lambda _{i}\left( n\right) -1\right) \right\vert ^{r}\mathbb{I}\left(
\left\vert \Lambda _{j}\left( n\right) /n-1\right\vert \leq \varepsilon
,\left\vert \Lambda _{i}\left( n\right) /n-1\right\vert \leq \varepsilon
\right) \right]  \notag \\
&&+\mathbb{E}\left[ \left\vert n^{1/2}\left( \Lambda _{j}\left( n\right)
/\Lambda _{i}\left( n\right) -1\right) \right\vert ^{r}\mathbb{I}\left(
\left\vert \Lambda _{j}\left( n\right) /n-1\right\vert >\varepsilon \cup
\left\vert \Lambda _{i}\left( n\right) /n-1\right\vert >\varepsilon \right) %
\right] .  \notag
\end{eqnarray}%
It is straightforward to verify (for example by computing moment generating
functions of the Gamma distribution) that
\begin{equation}
\sup_{n\geq 1}\mathbb{E}\left( n^{r/2}\left\vert \frac{\Lambda _{j}\left(
n\right) }{n}-1\right\vert ^{r}\right) <\infty  \label{B1aa}
\end{equation}%
for any $r>0$ and further, we can conclude that
\begin{eqnarray*}
&&\sup_{n\geq 1}\mathbb{E}\left[ \left\vert n^{1/2}\left( \Lambda _{j}\left(
n\right) /\Lambda _{i}\left( n\right) -1\right) \right\vert ^{r}\mathbb{I}%
\left( \left\vert \Lambda _{j}\left( n\right) /n-1\right\vert \leq
\varepsilon ,\left\vert \Lambda _{i}\left( n\right) /n-1\right\vert \leq
\varepsilon \right) \right] \\
&\leq &\sup_{n\geq 1}\mathbb{E}\left[ \left\vert n^{1/2}\left( \left( \left(
\frac{\Lambda _{j}\left( n\right) }{n}-1\right) -\left( \frac{\Lambda
_{i}\left( n\right) }{n}-1\right) \right) /\left( 1-\epsilon \right) \right)
\right\vert ^{r}\mathbb{I}\left( \left\vert \Lambda _{j}\left( n\right)
/n-1\right\vert \leq \varepsilon ,\left\vert \Lambda _{i}\left( n\right)
/n-1\right\vert \leq \varepsilon \right) \right] \\
&\leq &\frac{2^{r-1}}{\left( 1-\varepsilon \right) ^{r}}\sup_{n\geq 1}%
\mathbb{E}\left( n^{r/2}\left\vert \frac{\Lambda _{j}\left( n\right) }{n}%
-1\right\vert ^{r}\right) <\infty .
\end{eqnarray*}%
Then, because $\Lambda _{j}\left( n\right) /n$ (being the sum of $n-j+1$
i.i.d. random variables with finite moment generating function) satisfies
the large deviations principle (see, for instance, Chapter 2.2 in \cite%
{largedeviation}), we have
\begin{eqnarray}
&&\mathbb{E}\left[ \left\vert n^{1/2}\left( \Lambda _{j}\left( n\right)
/\Lambda _{i}\left( n\right) -1\right) \right\vert ^{r}\mathbb{I}\left(
\left\vert \Lambda _{j}\left( n\right) /n-1\right\vert >\varepsilon \right) %
\right]  \notag \\
&\leq &\sqrt{\mathbb{E}\left[ \left\vert n^{1/2}\left( \Lambda _{j}\left(
n\right) /\Lambda _{i}\left( n\right) -1\right) \right\vert ^{2r}\right] }%
\sqrt{\mathbb{P}\left( \left\vert \Lambda _{j}\left( n\right)
/n-1\right\vert >\varepsilon \right) }.  \label{AU3a}
\end{eqnarray}%
Because of our discussion involving the finiteness of the first two factors
in (\ref{AU0}), we can conclude that when $n>d+8r$,
\begin{eqnarray*}
\sqrt{\mathbb{E}\left[ \left\vert n^{1/2}\left( \Lambda _{j}\left( n\right)
/\Lambda _{i}\left( n\right) -1\right) \right\vert ^{2r}\right] } &\leq
&n^{r}\sqrt{1+\mathbb{E}\left[ \left( \frac{\Lambda _{j}\left( n\right) }{n}%
\times \frac{n}{\Lambda _{i}\left( n\right) }\right) ^{2r}\right] } \\
&\leq &n^{r}\sqrt{1+\sqrt{\mathbb{E}\left( \left( \Lambda _{j}\left(
n\right) /n\right) ^{4r}\right) }\times \sqrt{\mathbb{E}\left( \left(
n/\Lambda _{i}\left( n\right) \right) ^{4r}\right) }} \\
&\leq &C_r^{\prime }n^{r}.
\end{eqnarray*}%
Notice that from \ref{AU2} and \ref{AU22}, we have
\begin{equation*}
\sup_{n > d+8r} \mathbb{E}\left( \left( \Lambda _{j}\left(
n\right) /n\right) ^{4r}\right) \times \mathbb{E}\left( \left(
n/\Lambda _{i}\left( n\right) \right) ^{4r}\right)< \infty,
\end{equation*}
 which means $C_r^{\prime }$ is a numerical constant independent with $n$. Therefore, the first term in the
right hand side of (\ref{AU3a}) grows at rate $O\left( n^{r}\right) $, which
is polynomial, whereas the second term, due to the large deviations
principle invoked earlier converges exponentially fast to zero for each $%
\varepsilon >0$. This completes the first part of Lemma \ref{temp}.
\end{proof}

\begin{proof}[Proof of (2).]
For the second part of Lemma \ref{temp}, note that Lemma \ref{trace_ineqn}
implies%
\begin{equation*}
\left\vert \left\langle \sqrt{n}\left( \hat{\Sigma}_{n}-\Sigma _{0}\right) ,%
\hat{A}_{n}\right\rangle \right\vert ^{2}\leq \left\Vert \sqrt{n}\left( \hat{%
\Sigma}_{n}-\Sigma _{0}\right) \right\Vert _{F}^{2}\left\Vert \hat{A}%
_{n}\right\Vert _{F}^{2}.
\end{equation*}%
Then, for the uniform integrability of $\left\Vert \sqrt{n}\left( \hat{\Sigma%
}_{n}-\Sigma _{0}\right) \right\Vert _{F}^{2},$ we have
\begin{equation*}
\left\Vert \sqrt{n}\left( \hat{\Sigma}_{n}-\Sigma _{0}\right) \right\Vert
_{F}^{4}=\left\Vert C\sqrt{n}\left( \frac{1}{n}\sum_{i=1}^{n}\zeta
_{i}^{T}\zeta _{i}-I\right) C^{T}\right\Vert _{F}^{4}\leq \left\Vert \sqrt{n}%
\left( \frac{1}{n}\sum_{i=1}^{n}\zeta _{i}^{T}\zeta _{i}-I\right)
\right\Vert _{F}^{4}\left\Vert C\right\Vert _{F}^{8},
\end{equation*}%
by Lemma \ref{Lem_Aux_UI1}. We denote $\hat{\Psi}^{(n)}=\frac{1}{n}%
\sum_{i=1}^{n}\zeta _{i}^{T}\zeta _{i}.$ And by the Cauchy-Schwarz
inequality, we have%
\begin{equation*}
\left\Vert \sqrt{n}\left( \frac{1}{n}\sum_{i=1}^{n}\zeta _{i}^{T}\zeta
_{i}-I\right) \right\Vert _{F}^{4}=\left( \sum_{i,j}\left( \sqrt{n}\left(
\hat{\Psi}_{ij}^{(n)}-\delta _{ij}\right) \right) ^{2}\right) ^{2}\leq
d^{2}\sum_{i,j}\left( \sqrt{n}\left( \hat{\Psi}_{ij}^{(n)}-\delta
_{ij}\right) \right) ^{4},
\end{equation*}%
where $\delta _{ij}=\mathbb{I}\{i=j\}.$ Note that%
\begin{eqnarray*}
&&\mathbb{E}\left[ \sum_{i,j}\left( \sqrt{n}\left( \hat{\Psi}%
_{ij}^{(n)}-\delta _{ij}\right) \right) ^{4}\right] \\
&=&\mathbb{E}\left[ \sum_{i=1}^{d}\left( \sqrt{n}\left( \left( \frac{1}{n}%
\sum_{k=1}^{n}z_{ik}^{2}\right) -1\right) \right) ^{4}\right] +2\mathbb{E}%
\left[ \sum_{i=1}^{d}\sum_{j=1+1}^{d}\left( \sqrt{n}\left( \frac{1}{n}%
\sum_{k=1}^{n}z_{ik}z_{jk}\right) \right) ^{4}\right] \\
&=&d\mathbb{E}\left[ \left( \frac{1}{\sqrt{n}}\sum_{k=1}^{n}\left(
z_{ik}^{2}-1\right) \right) ^{4}\right] +d(d-1)\mathbb{E}\left[ \left( \frac{%
1}{\sqrt{n}}\sum_{k=1}^{n}z_{ik}z_{jk}\right) ^{4}\right],
\end{eqnarray*}%
where $z_{ik}\sim N(0,1)$ are i.i.d random variables. Further, direct
calculations give us
\begin{eqnarray*}
\mathbb{E}\left[ \left( \frac{1}{\sqrt{n}}\sum_{k=1}^{n}\left(
z_{ik}^{2}-1\right) \right) ^{4}\right] &=&\frac{1}{n^{2}}\left( \mathbb{E}%
\left[ \sum_{k=1}^{n}\left( z_{ik}^{2}-1\right) ^{4}\right] +6\frac{1}{n^{2}}%
\mathbb{E}\left[ \sum_{k=1}^{n}\sum_{l=k+1}^{n}\left( z_{ik}^{2}-1\right)
^{2}\left( z_{il}^{2}-1\right) ^{2}\right] \right) \\
&=&\frac{60}{n}+\frac{12(n-1)}{n}<\infty ,
\end{eqnarray*}%
and%
\begin{equation*}
\mathbb{E}\left[ \left( \frac{1}{\sqrt{n}}\sum_{k=1}^{n}z_{ik}z_{jk}\right)
^{4}\right] =\frac{1}{n^{2}}\mathbb{E}\left[
\sum_{k=1}^{n}z_{ik}^{4}z_{jk}^{4}\right] +6\frac{1}{n^{2}}\mathbb{E}\left[
\sum_{k=1}^{n}\sum_{l=k+1}^{n}z_{ik}^{2}z_{jk}^{2}z_{il}^{2}z_{jl}^{2}\right]
=\frac{9}{n}+\frac{3\left( n-1\right) }{n}<\infty .
\end{equation*}%
Therefore, we complete the uniform integrability of $\left\Vert \sqrt{n}%
\left( \hat{\Sigma}_{n}-\Sigma _{0}\right) \right\Vert _{F}^{2}.$

For $\left\Vert \hat{A}_{n}\right\Vert _{F}^{2},$ using the similar argument
with (\ref{g_func}), we have
\begin{equation}
\left\Vert \hat{A}_{n}\right\Vert _{F}^{2r}=\left\vert \text{tr}(\hat{A}%
_{n}^{2})\right\vert ^{r}=\frac{4^{r}\mathrm{tr}\left( \hat{\Sigma}_{n}^{-4}\right)
^{r}}{\text{tr}\left( \hat{\Sigma}_{n}^{-1}\right) ^{r}}\leq c_{2}\sqrt{%
\mathbb{E}\left[ \sum_{i=1}^{d}\left( \frac{W_{\left( i\right) }^{\left(
n\right) }}{n}\right) ^{2r}\right] }\sqrt{\mathbb{E}\left[
\sum_{i=1}^{d}\left( \frac{n}{W_{\left( i\right) }^{\left( n\right) }}%
\right) ^{8r}\right] }.  \label{A_n_bd}
\end{equation}
From the earlier bounds leading to the analysis of (\ref{AU0}), we complete
the uniform integrability of $\left\Vert \hat{A}_{n}\right\Vert _{F}^{2}$.
Hence, the second part of Lemma \ref{temp} follows.
\end{proof}

\begin{proof}[Proof of (3).]
For the third part of Lemma \ref{temp}, recall in the proof of Proposition \protect\ref{probCLT}(3), $f(\Sigma )=-\frac{2}{\sqrt{\text{tr}(\Sigma ^{-1})}}\Sigma ^{-2}$ and thus $\hat{A}_n=f\left(\hat{\Sigma}_n \right)$.

The argument is similar to that given to establish (\ref{AU2C}). We have
argued that $f\left( \cdot \right) $ is smooth around $\Sigma _{0}$, which
was the basis for the use of the delta method earlier in our argument.
Moreover, note that $\hat{\Sigma}_{n}$ satisfies a large deviations
principle. Therefore%
\begin{eqnarray*}
&&\left\vert \left\langle \sqrt{n}\left( \hat{\Sigma}_{n}-\Sigma _{0}\right)
,\sqrt{n}\left( \hat{A}_{n}-A_{0}\right) \right\rangle \right\vert \\
&=&\left\vert \left\langle \sqrt{n}\left( \hat{\Sigma}_{n}-\Sigma
_{0}\right) ,\sqrt{n}\left( f\left( \hat{\Sigma}_{n}\right) -A_{0}\right)
\right\rangle \right\vert \\
&=&\left\vert \left\langle \sqrt{n}\left( \hat{\Sigma}_{n}-\Sigma
_{0}\right) ,\sqrt{n}\left( f\left( \hat{\Sigma}_{n}\right) -A_{0}\right)
\right\rangle \right\vert \mathbb{I}\left( \left\Vert \hat{\Sigma}%
_{n}-\Sigma _{0}\right\Vert _{F}\leq \varepsilon \right) \\
&&+\left\vert \left\langle \sqrt{n}\left( \hat{\Sigma}_{n}-\Sigma
_{0}\right) ,\sqrt{n}\left( f\left( \hat{\Sigma}_{n}\right) -A_{0}\right)
\right\rangle \right\vert \mathbb{I}\left( \left\Vert \hat{\Sigma}%
_{n}-\Sigma _{0}\right\Vert _{F}>\varepsilon \right) .
\end{eqnarray*}%
By applying Lemma \ref{trace_ineqn} and the fact that $Df\left( \cdot
\right) $ is continuous around $\Sigma _{0}$ (see the expression of $%
Df\left( \cdot \right) $ in (\ref{Df})) we conclude that
\begin{eqnarray*}
&&\left\vert \left\langle \sqrt{n}\left( \hat{\Sigma}_{n}-\Sigma _{0}\right)
,\sqrt{n}\left( f\left( \hat{\Sigma}_{n}\right) -A_{0}\right) \right\rangle
\right\vert \mathbb{I}\left( \left\Vert \hat{\Sigma}_{n}-\Sigma
_{0}\right\Vert _{F}\leq \varepsilon \right) \\
&\leq &\sup_{\Sigma :\left\Vert \Sigma _{0}-\Sigma \right\Vert _{F}\leq
\varepsilon }\left\vert \left\langle \sqrt{n}\left( \hat{\Sigma}_{n}-\Sigma
_{0}\right) , Df\left( \Sigma \right) \left( \sqrt{n}\left( \hat{\Sigma%
}_{n}-\Sigma _{0}\right) \right)  \right\rangle \right\vert \\
&\leq &c_{0}\left\Vert \sqrt{n}\left( \hat{\Sigma}_{n}-\Sigma _{0}\right)
\right\Vert _{F}^{2},
\end{eqnarray*}%
where $c_{0}=\sup_{\Sigma :\left\Vert \Sigma _{0}-\Sigma \right\Vert
_{F}\leq \varepsilon }\left\Vert Df\left( \Sigma \right) \right\Vert
_{op}$ and $\left\Vert Df\left( \Sigma \right) \right\Vert
_{op}$ is the operator norm, which is defined by
\begin{equation*}
\left\Vert Df\left( \Sigma \right) \right\Vert
_{op}:=\sup\left\{ \left\Vert Df\left( \Sigma \right) A \right\Vert_F: A \in \mathbb{R}^{d\times d} \text{ with }\Vert A \Vert_F = 1    \right\}.
\end{equation*} Since $Df\left( \cdot
\right) $ is continuous around $\Sigma _{0}$, there exists sufficient small $\epsilon$ such that $c_0$ is finite. $\left\Vert \sqrt{n}\left( \hat{\Sigma}_{n}-\Sigma _{0}\right)
\right\Vert _{F}^{2}$ is proved to be uniformly integrable in the
second part of Lemma \ref{temp}. On the other hand, we have for $r>1$
\begin{eqnarray*}
&&\mathbb{E}\left[ \left\vert \left\langle \sqrt{n}\left( \hat{\Sigma}%
_{n}-\Sigma _{0}\right) ,\sqrt{n}\left( f\left( \hat{\Sigma}_{n}\right)
-A_{0}\right) \right\rangle \right\vert ^{r}\mathbb{I}\left( \left\Vert \hat{%
\Sigma}_{n}-\Sigma _{0}\right\Vert _{F}>\varepsilon \right) \right] \\
&\leq &\mathbb{E}\left[ \left\Vert \sqrt{n}\left( \hat{\Sigma}_{n}-\Sigma
_{0}\right) \right\Vert _{F}^{r}\left\Vert \sqrt{n}\left( f\left( \hat{\Sigma%
}_{n}\right) -A_{0}\right) \right\Vert _{F}^{r}\mathbb{I}\left( \left\Vert
\hat{\Sigma}_{n}-\Sigma _{0}\right\Vert _{F}>\varepsilon \right) \right] \\
&\leq &\sqrt{\mathbb{E}\left( \left\Vert \sqrt{n}\left( \hat{\Sigma}%
_{n}-\Sigma _{0}\right) \right\Vert _{F}^{2r}\right) }\mathbb{E}^{1/4}\left(
\left\Vert \sqrt{n}\left( f\left( \hat{\Sigma}_{n}\right) -A_{0}\right)
\right\Vert _{F}^{4r}\right) \left( \mathbb{P}\left( \left\Vert \hat{\Sigma}%
_{n}-\Sigma _{0}\right\Vert _{F}>\varepsilon \right) \right) ^{1/4}.
\end{eqnarray*}%
The proof of the second part of Lemma \ref{temp} shows $\mathbb{E}\left(
\left\Vert \sqrt{n}\left( \hat{\Sigma}_{n}-\Sigma _{0}\right) \right\Vert
_{F}^{2r}\right) <\infty$ when $r\leq 2.$ Further, we have argued throughout
the proof of the first part of Lemma \ref{temp} and the proof leading to (%
\ref{A_n_bd}) that%
\begin{equation*}
\mathbb{E}\left( \left\Vert \sqrt{n}\left( f\left( \hat{\Sigma}_{n}\right)
-A_{0}\right) \right\Vert _{F}^{4r}\right) \leq c_{3}n^{2r}\left( \mathbb{E}%
\left( \left\Vert \hat{A}_{n}\right\Vert ^{4r}\right) +\left\Vert
A_{0}\right\Vert ^{4r}\right) \leq O\left( n^{2r}\right) ,
\end{equation*}%
where $c_{3}$ is a numerical constant only related to $r$ and $d$. However,
the large deviations principle gives us $\mathbb{P}\left( \left\Vert \hat{%
\Sigma}_{n}-\Sigma _{0}\right\Vert _{F}>\varepsilon \right) =O\left( \exp
\left( -cn\right) \right) $ for some $c>0$. Therefore, using Lemmas \ref%
{trace_ineqn} and \ref{Lem_Aux_UI1} and the previous estimates we can
complete the last part of Lemma \ref{temp}.
\end{proof}

\subsection{Proof of Proposition \protect\ref{consistency}}

\begin{proof}
Let $X^{\ast }(\rho ):=\arg \min_{X\succ 0}\left\{ -\log \det X+\sup_{%
\mathbb{Q}\in \mathcal{P}_{\rho }^{\ast }}\mathbb{E}^{\mathbb{Q}}\left[
\left\langle \xi \xi ^{T},X\right\rangle \right] \right\} ,$ where
\begin{equation*}
\mathcal{P}_{\rho }^{\ast }=\left\{ \mathbb{Q}\sim \mathcal{N}\left(
0,\Sigma \right) \text{ for some }\Sigma :\mathbb{W}_{2}(\mathcal{N}\left(
0,\Sigma _{0}\right) ,\mathbb{Q})\leq \rho \right\} .
\end{equation*}%
Then, $\arg \min_{\rho \geq 0}\{\mathbb{E}[L(X^{\ast }(\rho ),\Sigma
_{0})]\}=0.$ Since $X_{n}^{\ast }(\rho )$ is a continuous function of $\hat{%
\Sigma}_{n},$ we have $X_{n}^{\ast }(\rho )\rightarrow $ $X^{\ast }(\rho )$
almost surely for all $\rho \geq 0$ by the continuous mapping theorem.
Furthermore, proof of Lemma \ref{temp} gives us $\mathbb{E}X_{n}^{\ast }(0)=%
\mathbb{E}\hat{\Sigma}_{n}^{-1}\rightarrow \Sigma _{0}^{-1}.$ And from
Lemmas 1 and 2 in \cite{cai2015law}, we have
\begin{equation*}
\left\vert \log \det (\hat{\Sigma}_{n}^{-1})-\log \det (\Sigma
_{0}^{-1})\right\vert =\left\vert \sum_{k=1}^{d}\log \left( \frac{1}{n}\chi
_{n-k-1}^{2}\right) \right\vert \rightarrow 0,
\end{equation*}%
where $\chi _{n}^{2},\ldots ,\chi _{n-p-1}^{2}$ are mutually independent $%
\chi ^{2}$ distribution with the degree of freedom $n,\ldots ,n-p+1$
respectively. Due to the uniform integrability of $\log \left(\frac{1}{n}\chi
_{n-k-1}^{2}\right)$, we conclude $\mathbb{E}\left[ \log \det \left( X_{n}^{\ast
}(0)\right) \right] \rightarrow \mathbb{E}\left[ \log \det \left( X^{\ast
}(0)\right) \right] .$

By (\ref{gamma_eqn}), (\ref{large_gamma_inequ}) and (\ref{x_eqn}), we have $%
x_{i}^{\ast }\leq \gamma ^{\ast }\leq d/\rho ^{2}$ and
\begin{eqnarray*}
\hat{x}_{i}^{\ast } &=&\gamma ^{\ast }\left[ 1-\frac{1}{2}\left( \sqrt{\hat{%
\lambda}_{i}^{2}\left( \gamma ^{\ast }\right) ^{2}+4\hat{\lambda}_{i}\gamma
^{\ast }}-\hat{\lambda}_{i}\gamma ^{\ast }\right) \right] \\
&\geq &\frac{1}{\rho ^{2}}\left( \sum_{i=1}^{d}\frac{1}{\hat{\lambda}%
_{i}\gamma ^{\ast }+2}\right) \frac{1}{\hat{\lambda}_{i}\gamma ^{\ast }+2} \\
&\geq &\frac{1}{\rho ^{2}}\left( \frac{1}{\hat{\lambda}_{i}\gamma ^{\ast }+2}%
\right) ^{2}.
\end{eqnarray*}%
Therefore, we have%
\begin{eqnarray*}
L(X_{n}^{\ast }(\rho ),\Sigma _{0}) &=&-\log \det (X_{n}^{\ast }(\rho
)\Sigma _{0})+\left\langle X_{n}^{\ast }(\rho ),\Sigma _{0}\right\rangle -d
\\
&=&-\log \det (X_{n}^{\ast }(\rho ))+\left\langle X_{n}^{\ast }(\rho
),\Sigma _{0}\right\rangle -\log \det (\Sigma _{0})-d \\
&\geq &d\log \left( \rho ^{2}/d\right) -\log \det (\Sigma _{0})-d.
\end{eqnarray*}%
Thus, $\rho _{n}\in \lbrack 0,C]$, for large enough $n$ and a large enough
constant $C.$ By proposition 3.5 in \cite{nguyen2018distributionally}, we
have $\log \det (X_{n}^{\ast }(\rho ))$ and $\left\langle X_{n}^{\ast }(\rho
),\Sigma _{0}\right\rangle $ decrease with $\rho .$ Then, $\mathbb{E}%
\left\langle X_{n}^{\ast }(\rho ),\Sigma _{0}\right\rangle \rightarrow
\mathbb{E}\left\langle X^{\ast }(\rho ),\Sigma _{0}\right\rangle $ since $%
\mathbf{0}\preceq X_{n}^{\ast }(\rho )\preceq X_{n}^{\ast }(0)=\hat{\Sigma}%
_{n}^{-1}$ and $\hat{\Sigma}_{n}^{-1}$ is uniformly integrable$.$ For $\log
\det (X_{n}^{\ast }(\rho )),$ the upper bound is given by
\begin{equation*}
\log \det (X_{n}^{\ast }(\rho ))\leq \log \det (\hat{\Sigma}_{n}^{-1}),
\end{equation*}%
and the lower bound is given by
\begin{eqnarray*}
\log \det (X_{n}^{\ast }(\rho )) &=&\log \left( \prod\limits_{i=1}^{d}\hat{x}%
_{i}^{\ast }\right) \\
&\geq &\log \left( \prod\limits_{i=1}^{d}\frac{1}{\rho ^{2}}\left( \frac{1}{%
\hat{\lambda}_{i}\gamma ^{\ast }+2}\right) ^{2}\right) \\
&\geq &-2d\log \rho -2\sum_{i=1}^{d}\left( \log \left( \hat{\lambda}%
_{i}\left( d/\rho ^{2}\right) +2\right) \right) \\
&\geq &-2d\log \rho -2\log \det \left( \left( d/\rho ^{2}\right) \hat{\Sigma}%
_{n}+2I\right) .
\end{eqnarray*}%
Due to the uniform integrability of $\log \det \left( \left( d/\rho
^{2}\right) \hat{\Sigma}_{n}+2I\right) $ and $\log \det (\hat{\Sigma}%
_{n}^{-1}),$ we have $\mathbb{E}\left[ \log \det (X_{n}^{\ast }(\rho ))%
\right] \rightarrow \mathbb{E}\left[ \log \det (X^{\ast }(\rho ))\right] .$
Finally, by the monotonicity of $\mathbb{E}\left[ \log \det (X_{n}^{\ast
}(\rho ))\right] $ and $\mathbb{E}\left\langle X_{n}^{\ast }(\rho ),\Sigma
_{0}\right\rangle ,$ we have $\mathbb{E}\left[ L(X_{n}^{\ast }(\rho ),\Sigma
_{0})\right] $ converges uniformly; thus, $\rho _{n}\rightarrow \arg
\min_{\rho \geq 0}\{\mathbb{E}[L(X^{\ast }(\rho ),\Sigma _{0})]\}=0.$
\end{proof}

\end{document}